\documentclass[11pt]{article}

\usepackage{amsfonts,amsmath,amsthm,amssymb}
\usepackage{bm}
\usepackage{mathrsfs}
\usepackage[colorlinks, citecolor=blue, urlcolor=blue]{hyperref}
\usepackage[numbers]{natbib}
\usepackage{geometry}
\usepackage[utf8]{inputenc}
\usepackage[T1]{fontenc}    
\usepackage{mathtools}
\usepackage{scalerel}
\usepackage{enumitem}

\newtheorem{theorem}{Theorem}[section]
\newtheorem{corollary}[theorem]{Corollary}
\newtheorem{lemma}[theorem]{Lemma}
\newtheorem{proposition}[theorem]{Proposition}
\newtheorem{assumption}[theorem]{Assumption}

\newtheorem{example}[theorem]{Example}
\theoremstyle{definition}
\newtheorem{definition}[theorem]{Definition}

\newtheorem{remark}[theorem]{\textbf{Remark}}
\numberwithin{equation}{section}

\DeclareMathOperator{\newtriangledown}{\triangledown}
\DeclareMathOperator*{\newbigtriangledown}{\!\scalerel*{\triangledown}{\bigvee}\!}

\begin{document}

\title{Inverting the Markovian projection for pure jump processes}
\author{
Martin Larsson\footnote{Department of Mathematical Sciences, Carnegie Mellon University, \texttt{larsson@cmu.edu}} \and
Shukun Long\footnote{Department of Mathematical Sciences, Carnegie Mellon University, \texttt{shukunl@alumni.cmu.edu}}
}
\maketitle

\begin{abstract}
Markovian projections arise in problems where we aim to mimic the one-dimensional marginal laws of an It\^o semimartingale by using another It\^o process with Markovian dynamics. In applications, Markovian projections are useful in calibrating jump-diffusion models with both local and stochastic features, leading to the study of the inversion problems. In this paper, we invert the Markovian projections for pure jump processes, which can be used to construct calibrated local stochastic intensity (LSI) models for credit risk applications. Such models are jump process analogues of the notoriously hard to construct local stochastic volatility (LSV) models used in equity modeling.
\end{abstract}

\section{Introduction}\label{sec:1}
Given a general It\^o semimartingale $X$, its Markovian projection is an It\^o process $\widehat{X}$ with Markovian type differential characteristics, i.e.\ they are functions of time and the process itself, such that the law of $\widehat{X}_t$ agrees with the law of $X_t$ for every $t \geq 0$. The concept of Markovian projections was first introduced by Gy\"ongy \cite{MR0833267}, where existence results were proved for continuous It\^o semimartingales using techniques inspired by Krylov \cite{MR0808203}. Later, Brunick and Shreve \cite{MR3098443} extended the work of \cite{MR0833267} by relaxing the assumptions therein and proving stronger mimicking results involving functionals of sample paths. On the other hand, Bentata and Cont \cite{bentata2012mimicking} studied Markovian projections for jump diffusion processes. They imposed relatively strong assumptions, but meanwhile they also showed uniqueness and Markov property of the mimicking processes. In our previous work \cite{MR4814246}, we independently constructed Markovian projections for It\^o semimartingales with jumps. The assumptions in \cite{MR4814246} are generally weaker compared to \cite{bentata2012mimicking}, at the cost of not guaranteeing properties beyond existence. 

More specifically, let $X$ be an $\mathbb{R}^d$-valued c\`adl\`ag It\^o semimartingale with differential characteristics $(\beta, \alpha, \kappa)$ (see Definition~\ref{def:diff_char}). Then, under suitable integrability and growth conditions, there exist measurable functions $b: \mathbb{R}_+ \times \mathbb{R}^d \to \mathbb{R}^d$, $a: \mathbb{R}_+ \times \mathbb{R}^d \to \mathbb{S}_+^d$, and a L\'evy transition kernel $k$ from $\mathbb{R}_+ \times \mathbb{R}^d$ to $\mathbb{R}^d$ such that for Lebesgue-a.e.\ $t \geq 0$,
\begin{subequations}
	\begin{align}
		b(t, X_{t-}) &= \mathbb{E}[\beta_t \,|\, X_{t-}],\label{eq:mp_condexp1}\\
		a(t, X_{t-}) &= \mathbb{E}[\alpha_t \,|\, X_{t-}],\label{eq:mp_condexp2}\\
		k(t, X_{t-}, d\xi) &= \mathbb{E}[\kappa_t(d\xi) \,|\, X_{t-}].\label{eq:mp_condexp3}
	\end{align}
\end{subequations}
Moreover, there exists an It\^o semimartingale $\widehat{X}$ with Markovian type differential characteristics $b(t, \widehat{X}_{t-})$, $a(t, \widehat{X}_{t-})$ and $k(t, \widehat{X}_{t-}, d\xi)$ such that for every $t \geq 0$, the law of $\widehat{X}_t$ agrees with the law of $X_t$. The process $\widehat{X}$ is the so-called Markovian projection of $X$. For a more rigorous statement of this result, the readers may refer to \cite{MR4814246}, Theorem~3.2.

In applications, Markovian projections usually occur in the inversion problems. Let us say we begin with a relatively simple process $\widehat{X}$. The objective is to construct a process $X$ with richer dynamics, while preserving the one-dimensional marginal laws. If we can find a more complex process $X$ whose Markovian projection is $\widehat{X}$, then we are done as the marginal law constraints are automatically satisfied. This is often referred to as ``inverting the Markovian projection''. One of the most well-known applications of Markovian projections is the calibration of the local stochastic volatility (LSV) model in mathematical finance (see e.g.\ \cite{andersen2010interest}, Appendix~A, \cite{MR3155635}, Chapter~11, and the references therein). In an LSV model, the risk-neutral dynamics of the stock price is modeled via the following SDE (assuming zero interest rate and dividend yield):
\begin{equation}\label{eq:LSV0}
	dS_t = \eta_t \sigma(t, S_t) S_t \,dB_t,
\end{equation}
where $\eta$ is the stochastic volatility, $\sigma$ is a function to be calibrated, and $B$ is a Brownian motion. One requires exact calibration of the LSV model to European call option prices (which depend on the one-dimensional marginal laws of the stock price). By the seminal work of Dupire \cite{dupire1994pricing}, we have such exact calibration in the local volatility (LV) model:
\begin{equation*}
	d\widehat{S}_t = \sigma_{\mathrm{Dup}}(t, \widehat{S}_t) \widehat{S}_t \,d\widehat{B}_t,\quad
	\sigma_{\mathrm{Dup}}^2(t, K) \coloneqq \frac{\partial_t C(t, K)}{(1/2) K^2 \partial_{KK} C(t, K)},
\end{equation*}
where $\widehat{B}$ is a Brownian motion, and $C(t, K)$ is the European call option price with maturity $t$ and strike $K$. Thus, we only need $\widehat{S}$ to be a Markovian projection of $S$. Based on \eqref{eq:mp_condexp2} (here $\alpha_t = \eta_t^2 \sigma^2(t, S_t) S_t^2$), one can choose
\begin{equation}\label{eq:leverage}
	\sigma^2(t, x) \coloneqq \frac{\sigma_{\mathrm{Dup}}^2(t, x)}{\mathbb{E}[\eta_t^2 \,|\, S_t = x]}.
\end{equation}
Plugging \eqref{eq:leverage} into \eqref{eq:LSV0} yields the McKean--Vlasov type SDE
\begin{equation}\label{eq:LSV}
	dS_t = \frac{\eta_t}{\sqrt{\mathbb{E}[\eta_t^2 \,|\, S_t]}} \sigma_{\mathrm{Dup}}(t, S_t) S_t \,dB_t.
\end{equation}

However, the SDE \eqref{eq:LSV} is notoriously difficult to solve, and finding a solution in full generality still remains an open problem. Partial results exist when $\eta$ takes the form $f(Y)$. For instance, Abergel and Tachet \cite{MR2629564} established the short-time existence of solutions when $Y$ is a multi-dimensional diffusion process. Jourdain and Zhou \cite{MR4086600} showed the weak existence of solutions with $Y$ being a finite-state jump process, provided that $f$ satisfies a structural condition. Lacker, Shkolnikov and Zhang \cite{MR4152640} proved the strong existence and uniqueness of stationary solutions, when $\sigma_{\mathrm{Dup}}$ has no dependency on $t$ and $Y$ solves a time homogeneous SDE driven by an independent Brownian motion. Djete \cite{djete2022non} recently provided a more general existence result when $\eta$ takes the form $f(t, S, Y)$ for a specific class of functions $f$, and $Y$ solves some SDE. In his setting, the SDE of $Y$ is allowed to have time dependent coefficients and can be driven by a correlated Brownian motion.

In this paper, we study a jump process analogue of the above problem. We invert the Markovian projections for pure jump processes. Our work is motivated by the calibration of the local stochastic intensity (LSI) models (see \cite{MR3481308}), which often appear in credit risk applications. In an LSI model, the number of defaults is modeled via a counting process $X$ whose intensity takes the form $\eta_t \lambda(t, X_{t-})$, with $\eta$ representing the stochastic intensity and $\lambda$ being a function to be calibrated. That is to say, the process
\begin{equation*}
	X_t - \int_0^t \eta_s \lambda(s, X_{s-}) \,ds,\quad t \geq 0,
\end{equation*}
is a martingale. Similar to the LSV model, here we want to align the one-dimensional marginal laws of the LSI model with the local intensity (LI) model, in which there is exact calibration to collateralized debt obligation (CDO) tranche prices (see \cite{schonbucher2005portfolio}). Recall that in an LI model, defaults are modeled via a counting process $\widehat{X}$ with intensity of the form $\lambda_{\mathrm{LI}}(t, \widehat{X}_{t-})$. Based on \eqref{eq:mp_condexp3} (here $\kappa_t(d\xi) = \eta_t \lambda(t, X_{t-}) \delta_1(d\xi)$), one can choose
\begin{equation*}
	\lambda(t, x) = \frac{\lambda_{\mathrm{LI}}(t, x)}{\mathbb{E}[\eta_t \,|\, X_{t-} = x]},
\end{equation*}
which leads to the McKean--Vlasov type martingale problem:
\begin{equation*}
	\biggl(X_t - \int_0^t \frac{\eta_s}{\mathbb{E}[\eta_s \,|\, X_{s-}]} \lambda_{\mathrm{LI}}(s, X_{s-}) \,ds\biggr)_{t \geq 0}
	\text{ is a martingale}.
\end{equation*}
Under this setting, $\widehat{X}$ is a Markovian projection of $X$. In particular, when $\widehat{X}$ is a Poisson process (i.e.\ $\lambda_{\mathrm{LI}}$ is a deterministic function of time $t$), we call $X$ a \emph{fake Poisson process}. This notion is analogous to the so-called fake Brownian motion.

Alfonsi, Labart and Lelong \cite{MR3481308} constructed solutions to the LSI model for $\eta_t = f(Y_t)$, where $Y$ is either an $\mathbb{N}$-valued Markov chain or an $\mathbb{R}$-valued jump diffusion process solving an SDE of the following type:
\begin{equation*}
	dY_t = b(t, X_{t-}, Y_{t-}) \,dt + \sigma(t, X_{t-}, Y_{t-}) \,dB_t + \gamma(t, X_{t-}, Y_{t-}) \,dX_t.
\end{equation*}
For the case of real-valued $Y$, they proved weak existence and uniqueness using a system of interacting particles and the Banach fixed point theorem respectively. They also showed strong existence and uniqueness.

In this paper, we establish the existence of solutions to the LSI model under weaker regularity assumptions, while our $\eta$ is an exogenously given process rather than being in the aforementioned feedback form involving $X$. Our main tools are the so-called Cox construction and suitable fixed point theorems. The existence proved in this paper is not in the classical weak or strong sense. It is not weak since we work on exactly the same probability space where $\eta$ is given. On the other hand, it is not strong since we have to enlarge the initial filtration in the construction of $X$. Even in this slightly nonstandard situation, one may ask about uniqueness in law. This requires additional assumptions, which turn out to imply that the solution must arise from a Cox construction. Another novelty of this paper is that we extend the jump size of $X$ to follow any discrete law with finite first moment. As we will see in the proof, allowing negative jump sizes brings new difficulties to the problem.

This paper is organized as follows. In Section~\ref{sec:2}, we gather all relevant information about the Cox construction. In Section~\ref{sec:3}, we prove our main results on inverting the Markovian projections for counting processes. We also discuss conditions for uniqueness in law. In Section~\ref{sec:4}, we extend the results in Section~\ref{sec:3} to general discrete jump size distributions. Throughout this paper, we follow the conventions listed below:
\begin{itemize}[nosep]
	\item $\mathbb{R}_+ = [0, \infty)$.
	
	\item $\mathbb{N}$ ($\mathbb{N}^*$) is the set of natural numbers including (excluding) $0$.
	
	\item $\mathbb{S}_+^d$ is the set of symmetric positive semi-definite $d \times d$ real matrices.
	
	\item Let $(a_i)_{i \in \mathbb{N}}$ be a sequence. Define $\sum_{i \in \varnothing} a_i \coloneqq 0$ and $\prod_{i \in \varnothing} a_i \coloneqq 1$.
	
	\item All semimartingales have c\`adl\`ag sample paths.
\end{itemize}

\section{Preliminary Results}\label{sec:2}
In this section, we review the Cox construction, which is the foundation for proving our main results. We briefly review some well-known results in Aksamit and Jeanblanc \cite{MR3729407}. Beyond that, we also complement the theory to serve our purposes.

\subsection{Basic Concepts}
We introduce some notation here, which will be used throughout this paper. First we need a notion of ``union'' for filtrations.

\begin{definition}\label{def:vee}
	Let $\mathcal{F}$, $\mathcal{G}$ and $(\mathcal{F}^i)_{i \in I}$ be $\sigma$-algebras on some space. Let $\mathbb{F} = (\mathcal{F}_t)_{t \geq 0}$, $\mathbb{G} = (\mathcal{G}_t)_{t \geq 0}$ and $(\mathbb{F}^i = (\mathcal{F}^i_t)_{t \geq 0})_{i \in I}$ be filtrations on the same space. We denote
	\begin{enumerate}[label=(\roman*), nosep]
		\item $\mathcal{F} \vee \mathcal{G} \coloneqq \sigma(\mathcal{F} \cup \mathcal{G})$ and $\bigvee_{i \in I} \mathcal{F}^i \coloneqq \sigma(\bigcup_{i \in I} \mathcal{F}^i)$,
		
		\item $\mathbb{F} \lor \mathbb{G} \coloneqq (\mathcal{F}_t \lor \mathcal{G}_t)_{t \geq 0}$ and $\bigvee_{i \in I} \mathbb{F}^i \coloneqq (\bigvee_{i \in I} \mathcal{F}^i_t)_{t \geq 0}$,
		
		\item $\mathbb{F} \vee \mathcal{G} \coloneqq (\mathcal{F}_t \lor \mathcal{G})_{t \geq 0}$. 
	\end{enumerate}
\end{definition}

There is a technical issue with the filtration $\mathbb{F} \lor \mathbb{G}$. Recall that many classical results in stochastic analysis rely on the completeness and right-continuity of filtrations (the so-called \emph{usual conditions}), so we always prefer to work with such ``nice'' filtrations. The problem is that even when both $\mathbb{F}$ and $\mathbb{G}$ satisfy the usual conditions, the right-continuity of $\mathbb{F} \lor \mathbb{G}$ may break. Thus, we need a slightly different notion of ``union'' for filtrations.

\begin{definition}
	Let $\mathbb{F} = (\mathcal{F}_t)_{t \geq 0}$ be a filtration. The \emph{right-continuous regularization} of $\mathbb{F}$ is the filtration defined by $(\bigcap_{s > t} \mathcal{F}_s)_{t \geq 0}$.
\end{definition}

Note that by construction, the right-continuous regularization of $\mathbb{F}$ is the smallest right-continuous filtration that includes $\mathbb{F}$. Now we can define the proper notion of ``union'' for filtrations.

\begin{definition}\label{def:tridwn}
	Under the setting of Definition~\ref{def:vee}, we define $\mathbb{F} \newtriangledown \mathbb{G}$, $\newbigtriangledown_{i \in I} \mathbb{F}^i$ and $\mathbb{F} \newtriangledown \mathcal{G}$ as the right-continuous regularization of $\mathbb{F} \lor \mathbb{G}$, $\bigvee_{i \in I} \mathbb{F}^i$ and $\mathbb{F} \vee \mathcal{G}$, respectively.
\end{definition}

In this paper, we prefer to work with $\newtriangledown$ rather than $\vee$. This is purely for technical reasons (i.e.\ right-continuity), and interchangeably using these two operators does not affect our main results. One important property satisfied by the operator $\newtriangledown$ is associativity, just like the operator $\vee$.

\begin{proposition}\label{prp:tridwn}
	Let $(\mathbb{F}^i)_{i \in I}$ be filtrations on some space. If $(I_j)_{j \in J}$ is a partition of the index set $I$, then $\newbigtriangledown_{i \in I} \mathbb{F}^i = \newbigtriangledown_{j \in J} (\newbigtriangledown_{i \in I_j} \mathbb{F}^i)$.
\end{proposition}

The proof of Proposition~\ref{prp:tridwn} is elementary, so we omit it here. As a consequence, for finitely many filtrations $(\mathbb{F}^i)_{i=1}^n$, there is no ambiguity whether we define $\newbigtriangledown_{i=1}^n \mathbb{F}^i$ according to Definition~\ref{def:tridwn}, or inductively by $((\mathbb{F}^1 \newtriangledown \mathbb{F}^2) \newtriangledown \cdots) \newtriangledown \mathbb{F}^n$.

Next, we discuss another important notion called \emph{immersion}. Let $\mathbb{F} = (\mathcal{F}_t)_{t \geq 0}$ and $\mathbb{G} = (\mathcal{G}_t)_{t \geq 0}$ be filtrations. Similar to set inclusions, we write $\mathbb{F} \subseteq \mathbb{G}$ if $\mathcal{F}_t \subseteq \mathcal{G}_t$ for all $t \geq 0$. In such cases, if $M$ is a $\mathbb{G}$-martingale, then it is also an $\mathbb{F}$-martingale as long as it is $\mathbb{F}$-adapted. The converse direction is not true in general, i.e.\ an $\mathbb{F}$-martingale is not necessarily a $\mathbb{G}$-martingale. This leads to the following definition.

\begin{definition}[Immersion]
	Let $\mathbb{F}$ and $\mathbb{G}$ be filtrations with $\mathbb{F} \subseteq \mathbb{G}$. We say $\mathbb{F}$ is \emph{immersed} in $\mathbb{G}$ if every $\mathbb{F}$-martingale is a $\mathbb{G}$-martingale, and we denote this by $\mathbb{F} \hookrightarrow \mathbb{G}$.
\end{definition}

The immersion property is also known as \emph{Hypothesis $\mathscr{H}$} in the filtration enlargement literature. This notion seems to be first introduced in Br\'emaud and Yor \cite{MR511775}. For equivalent characterizations of this property, one may refer to \cite{MR511775}, Theorem~3, \cite{MR3729407}, Theorem~3.2 (and the remarks preceding it), and \cite{MR3618140}, Theorem~2. See also \cite{beiglbock2018denseness} for a new characterization in the context of progressively enlarged filtrations. One important application of immersion is the credit risk modeling in mathematical finance (see e.g.\ \cite{MR1802597} and \cite{MR2648600}).

Immersion is not easy to check in general. There are some special cases where immersion holds. Consider filtrations $\mathbb{F} \subseteq \mathbb{G} \subseteq \mathbb{H}$. If $\mathbb{F} \hookrightarrow \mathbb{G}$ and $\mathbb{G} \hookrightarrow \mathbb{H}$, then obviously $\mathbb{F} \hookrightarrow \mathbb{H}$. Conversely, if $\mathbb{F} \hookrightarrow \mathbb{H}$, then we have $\mathbb{F} \hookrightarrow \mathbb{G}$ (there is no guarantee that $\mathbb{G} \hookrightarrow \mathbb{H}$). Some less trivial cases of immersion are listed in the lemma below.

\begin{lemma}\label{lem:immerse}
	Let $\mathbb{F}$, $\mathbb{G}$ and $(\mathbb{F}^i)_{i \in \mathbb{N}}$ be filtrations on some space.
	\begin{enumerate}[label=(\roman*), nosep]
		\item If $\mathbb{G}$ is the right-continuous regularization of $\mathbb{F}$, then $\mathbb{F} \hookrightarrow \mathbb{G}$.
		
		\item If $\mathbb{F}^{i-1} \hookrightarrow \mathbb{F}^i$ for all $i \in \mathbb{N^*}$, then $\mathbb{F}^k \hookrightarrow \newbigtriangledown_{i=0}^\infty \mathbb{F}^i$ for all $k \in \mathbb{N}$.
		
		\item If $Z$ is a random variable independent of $\mathbb{F}$, then $\mathbb{F} \hookrightarrow \mathbb{F} \newtriangledown \sigma(Z)$.
	\end{enumerate}
\end{lemma}

\begin{proof}
	For (i), write $\mathbb{F} = (\mathcal{F}_t)_{t \geq 0}$ and $\mathbb{G} = (\mathcal{G}_t)_{t \geq 0}$. Let $M$ be an $\mathbb{F}$-martingale, and let $0 \leq s < r < t$. By the $\mathbb{F}$-martingale property of $M$, we have
	\begin{equation*}
		\mathbb{E}[M_t \,|\, \mathcal{F}_r] = M_r.
	\end{equation*}
	Send $r \downarrow s$ along a sequence. On the left-hand side we use L\'evy's downward convergence theorem, and on the right-hand side we use the right-continuity of the sample paths of $M$ (which is our convention). This yields $\mathbb{E}[M_t \,|\, \mathcal{G}_s] = M_s$, so $M$ is a $\mathbb{G}$-martingale.
	
	For (ii), write $\mathbb{F}^i = (\mathcal{F}^i_t)_{t \geq 0}$ for $i \in \mathbb{N}$. Let $M$ be an $\mathbb{F}^k$-martingale, and let $0 \leq s < t$. Note that $\mathbb{F}^k \hookrightarrow \mathbb{F}^n$ for all $n \geq k$, so by the $\mathbb{F}^n$-martingale property of $M$, we have
	\begin{equation*}
		\mathbb{E}[M_t \,|\, \mathcal{F}^n_s] = M_s.
	\end{equation*}
	Send $n \uparrow \infty$. L\'evy's upward convergence theorem yields $\mathbb{E}[M_t \,|\, \bigvee_{i=0}^\infty \mathcal{F}^i_s] = M_s$, so $M$ is a ($\bigvee_{i=0}^\infty \mathbb{F}^i$)-martingale. This shows $\mathbb{F}^k \hookrightarrow \bigvee_{i=0}^\infty \mathbb{F}^i$, and applying (i) proves the result.
	
	For (iii), see \cite{MR3729407}, Proposition~1.22 and apply (i).
\end{proof}

Finally, for a filtration $\mathbb{F} = (\mathcal{F}_t)_{t \geq 0}$, we write $\mathcal{F}_\infty \coloneqq \bigvee_{t \geq 0} \mathcal{F}_t$.

\subsection{The Cox Construction}\label{subsec:cox}
From now on, we let $(\Omega, \mathcal{F}, \mathbb{P})$ be a probability space and $\mathbb{F} = (\mathcal{F}_t)_{t \geq 0}$ be a filtration on that space satisfying the usual conditions, i.e.\ $\mathbb{F}$ is right-continuous and $\mathcal{F}_0$ contains all $\mathbb{P}$-null sets. Unless stated otherwise, we always work on this filtered probability space. Now we discuss the Cox construction. The first lemma states the case of one jump time and constant jump size, which was already proved in \cite{MR3729407}. We summarize the results below.

\begin{lemma}\label{lem:cox}
	Let $\lambda$ be a nonnegative progressively measurable process, and denote $\Lambda_t \coloneqq \int_0^t \lambda_s \,ds$. Suppose that $\Lambda_t < \infty$ for all $t \geq 0$ and $\Lambda_\infty = \infty$, $\mathbb{P}$-a.s. Moreover, suppose that $\Omega$ supports an $\operatorname{Exp}(1)$ random variable $E$ which is independent of $\mathcal{F}_\infty$. Define the random time
	\begin{equation*}
		\tau \coloneqq \inf\{t > 0: \Lambda_t \geq E\},
	\end{equation*}
	and define the filtration $\mathbb{A} = (\mathcal{A}_t)_{t \geq 0}$ by $\mathcal{A}_t \coloneqq \sigma(\{\tau \leq s\}; s \leq t)$. Let $\mathbb{G} \coloneqq \mathbb{F} \newtriangledown \mathbb{A}$. Then, the process $(\bm{1}_{\{\tau \leq t\}} - \Lambda_{\tau \land t})_{t \geq 0}$ is a $\mathbb{G}$-martingale, and we have $\mathbb{F} \hookrightarrow \mathbb{G}$.
\end{lemma}

\begin{proof}
	See \cite{MR3729407}, Corollary~2.26 and Lemma~2.28.
\end{proof}

\begin{remark}
	By the assumptions on $\Lambda$, we have $\tau < \infty$ $\mathbb{P}$-a.s. On the other hand, the random time $\tau$ is not an $\mathbb{F}$-stopping time, since $E$ is independent of $\mathcal{F}_\infty$. However, $\tau$ is a $\mathbb{G}$-stopping time. Indeed, $\mathbb{G}$ is the smallest right-continuous filtration that includes $\mathbb{F}$ and makes $\tau$ a stopping time.
\end{remark}

Next we extend Lemma~\ref{lem:cox} to the case of random jump size, instead of constant $1$ as for the indicator process $(\bm{1}_{\{\tau \leq t\}})_{t \geq 0}$. For convenience, we call a random variable $J$ \emph{non-vanishing}, if $\mathbb{P}(J \neq 0) = 1$.

\begin{lemma}\label{lem:cox_J}
	Let $\lambda$, $\Lambda$, $E$ and $\tau$ be given by Lemma~\ref{lem:cox}. Furthermore, suppose that $\Omega$ supports a non-vanishing integrable random variable $J$ which is independent of $\mathcal{F}_\infty \lor \sigma(E)$. Define the filtration $\widetilde{\mathbb{A}} = (\widetilde{\mathcal{A}}_t)_{t \geq 0}$ by $\widetilde{\mathcal{A}}_t \coloneqq \sigma(J\bm{1}_{\{\tau \leq s\}}; s \leq t)$. Let $\widetilde{\mathbb{G}} \coloneqq \mathbb{F} \newtriangledown \widetilde{\mathbb{A}}$. Then, the process $(J\bm{1}_{\{\tau \leq t\}} - \mathbb{E}[J]\Lambda_{\tau \land t})_{t \geq 0}$ is a $\widetilde{\mathbb{G}}$-martingale, and we have $\mathbb{F} \hookrightarrow \widetilde{\mathbb{G}}$.
\end{lemma}

\begin{proof}
	Let $\mathbb{A}$ and $\mathbb{G}$ be the filtrations defined in Lemma~\ref{lem:cox}, i.e.\ $\mathbb{A} = (\mathcal{A}_t)_{t \geq 0}$ with $\mathcal{A}_t = \sigma(\{\tau \leq s\}; s \leq t)$, and $\mathbb{G} = \mathbb{F} \newtriangledown \mathbb{A}$. We notice that $\{\tau \leq s\} = \{J\bm{1}_{\{\tau \leq s\}} \neq 0\}$ since $J$ is non-vanishing. This implies that $\mathcal{A}_t \subseteq \widetilde{\mathcal{A}}_t \subseteq \mathcal{A}_t \vee \sigma(J)$ for all $t \geq 0$, i.e.\ $\mathbb{A} \subseteq \widetilde{\mathbb{A}} \subseteq \mathbb{A} \vee \sigma(J)$. From this we get the filtration inclusion $\mathbb{G} \subseteq \widetilde{\mathbb{G}} \subseteq \mathbb{G} \newtriangledown \sigma(J)$. By Lemma~\ref{lem:immerse}~(iii) and Lemma~\ref{lem:cox}, we have $\mathbb{F} \hookrightarrow \mathbb{G} \hookrightarrow \mathbb{G} \newtriangledown \sigma(J)$. This suffices to show $\mathbb{F} \hookrightarrow \mathbb{G} \hookrightarrow \widetilde{\mathbb{G}}$.
	
	Next, to show that the process $(J\bm{1}_{\{\tau \leq t\}} - \mathbb{E}[J]\Lambda_{\tau \land t})_{t \geq 0}$ is a $\widetilde{\mathbb{G}}$-martingale, let us decompose it into two parts:
	\begin{equation}\label{eq:cox_decomp}
		J\bm{1}_{\{\tau \leq t\}} - \mathbb{E}[J]\Lambda_{\tau \land t}
		= (J - \mathbb{E}[J]) \bm{1}_{\{\tau \leq t\}} + \mathbb{E}[J] (\bm{1}_{\{\tau \leq t\}} - \Lambda_{\tau \land t}).
	\end{equation}
	The second term in \eqref{eq:cox_decomp} is a $\mathbb{G}$-martingale according to Lemma~\ref{lem:cox}, thus a $\widetilde{\mathbb{G}}$-martingale by immersion. For the first term in \eqref{eq:cox_decomp}, we only need to show that it is a ($\mathbb{F} \vee \widetilde{\mathbb{A}}$)-martingale due to Lemma~\ref{lem:immerse}~(i). Let $0 \leq s < t$. Our goal is to show
	\begin{equation*}
		\mathbb{E}[(J - \mathbb{E}[J]) \bm{1}_{\{\tau \leq t\}} \,|\, \mathcal{F}_s \vee \widetilde{\mathcal{A}}_s]
		= (J - \mathbb{E}[J]) \bm{1}_{\{\tau \leq s\}},
	\end{equation*}
	or equivalently $\mathbb{E}[(J - \mathbb{E}[J]) \bm{1}_{\{s < \tau \leq t\}} \,|\, \mathcal{F}_s \vee \widetilde{\mathcal{A}}_s] = 0$. By Dynkin's $\pi$-$\lambda$ theorem, it suffices to show 
	\begin{equation}\label{eq:cox_mtg}
		\mathbb{E}\Biggl[(J - \mathbb{E}[J]) \bm{1}_{\{s < \tau \leq t\}} \bm{1}_A \prod_{i=1}^n \bm{1}_{\{J\bm{1}_{\{\tau \leq s_i\}} \in B_i\}}\Biggr] = 0,
	\end{equation}
	for all $A \in \mathcal{F}_s$, $B_1, ..., B_n \in \mathcal{B}(\mathbb{R})$, $0 \leq s_1 < \cdots < s_n \leq s$ and $n \in \mathbb{N}^*$. However, since $J$ is independent of $(\mathcal{F}_\infty, \tau)$, the left-hand side of \eqref{eq:cox_mtg} simplifies to
	\begin{equation*}
		\mathbb{E}\Biggl[(J - \mathbb{E}[J]) \bm{1}_{\{s < \tau \leq t\}} \bm{1}_A \prod_{i=1}^n \bm{1}_{\{0 \in B_i\}}\Biggr]
		= \prod_{i=1}^n \bm{1}_{\{0 \in B_i\}} \mathbb{E}[J - \mathbb{E}[J]] \mathbb{E}[\bm{1}_{\{s < \tau \leq t\}} \bm{1}_A] = 0,
	\end{equation*}
	which finishes the proof.
\end{proof}

\begin{remark}\label{rem:cox_h(J)}
	Under the setting of Lemma~\ref{lem:cox_J}, for any measurable function $h: \mathbb{R} \to \mathbb{R}$ such that $h(J)$ is integrable, one can show that the process $(h(J)\bm{1}_{\{\tau \leq t\}} - \mathbb{E}[h(J)]\Lambda_{\tau \land t})_{t \geq 0}$ is a $\widetilde{\mathbb{G}}$-martingale. Indeed, since $J$ is non-vanishing, without loss of generality we may assume $h(0) = 0$. This gives us $h(J)\bm{1}_{\{\tau \leq t\}} = h(J\bm{1}_{\{\tau \leq t\}})$, so $\widetilde{\mathbb{G}}$-adaptedness holds. Then, the martingale property can be proved in exactly the same manner as above.
\end{remark}

We can further extend Lemma~\ref{lem:cox_J} by considering infinitely many jumps. Below is the main theorem of this section.

\begin{theorem}[Cox Construction]\label{thm:cox}
	Let $(\lambda^k)_{k \in \mathbb{N}^*}$ be a sequence of measurable processes, satisfying $c \leq \lambda^k \leq C$ $(\mathbb{P} \otimes dt)$-a.e., $\forall\, k \in \mathbb{N}^*$, for some constants $0 < c < C$. Suppose that $\Omega$ supports a sequence of i.i.d.\ $\operatorname{Exp}(1)$ random variables $(E_k)_{k \in \mathbb{N}^*}$ and a sequence of i.i.d.\ non-vanishing integrable random variables $(J_k)_{k \in \mathbb{N}^*}$ such that $\mathcal{F}_\infty$, $(E_k)_{k \in \mathbb{N}^*}$ and $(J_k)_{k \in \mathbb{N}^*}$ are independent. Let $\tau_0 \coloneqq 0$. For $k \in \mathbb{N}^*$, inductively define the random time
	\begin{equation*}
		\tau_k \coloneqq \inf\biggl\{t > \tau_{k-1}: \int_{\tau_{k-1}}^t \lambda^k_s \,ds \geq E_k\biggr\},
	\end{equation*}
	and define the filtration $\mathbb{A}^k = (\mathcal{A}^k_t)_{t \geq 0}$ by $\mathcal{A}^k_t \coloneqq \sigma(J_k \bm{1}_{\{\tau_k \leq s\}}; s \leq t)$. Let $\mathbb{G}^0 \coloneqq \mathbb{F}$, $\mathbb{G}^k \coloneqq \mathbb{F} \newtriangledown (\newbigtriangledown_{i=1}^k \mathbb{A}^i)$ for $k \in \mathbb{N}^*$ and $\mathbb{G} \coloneqq \mathbb{F} \newtriangledown (\newbigtriangledown_{i=1}^\infty \mathbb{A}^i)$. Suppose that each $\lambda^k$ is progressively measurable with respect to $\mathbb{G}^{k-1}$. Then, the pure jump process $X_t \coloneqq \sum_{k=1}^\infty J_k \bm{1}_{\{\tau_k \leq t\}}$, $t \geq 0$, is well-defined and the process
	\begin{equation*}
		M_t \coloneqq X_t - \mathbb{E}[J_1] \int_0^t \sum_{k=1}^\infty \bm{1}_{\{\tau_{k-1} < s \leq \tau_k\}} \lambda^k_s \,ds,\quad t \geq 0,
	\end{equation*}
	is a $\mathbb{G}$-martingale.
\end{theorem}

\begin{proof}
	First we show that $X$ is well-defined. Let $\sigma_k \coloneqq C^{-1} \sum_{i=1}^k E_i$ and $\widetilde{\sigma}_k \coloneqq c^{-1} \sum_{i=1}^k E_i$ for $k \in \mathbb{N}$. Since each $\lambda^k$ is bounded from above and below by $C$ and $c$, it is easy to check that $\sigma_k \leq \tau_k \leq \widetilde{\sigma}_k$ for all $k$ $\mathbb{P}$-a.s., i.e.\ the sequence $(\tau_k)_{k \in \mathbb{N}}$ is ``sandwiched'' between the jump times of two Poisson processes with intensities $C$ and $c$. This shows that $(\tau_k)_{k \in \mathbb{N}}$ is a sequence of finite $\mathbb{G}$-stopping times that increases to infinity $\mathbb{P}$-a.s. In particular, $X_t(\omega)$ is a finite sum for each $t \geq 0$ and $\mathbb{P}$-a.s.\ $\omega \in \Omega$.
	
	Next for each $k \in \mathbb{N}^*$, we define the process
	\begin{equation*}
		M^k_t \coloneqq J_k \bm{1}_{\{\tau_k \leq t\}} - \mathbb{E}[J_k] \int_{\tau_{k-1} \land t}^{\tau_k \land t} \lambda^k_s \,ds,\quad t \geq 0,
	\end{equation*}
	and show that it is a $\mathbb{G}$-martingale. Take $(\Omega, \mathcal{F}, \mathbb{G}^{k-1}, \mathbb{P})$ as our underlying filtered probability space, on which $\mathbb{G}^{k-1}$, $E_k$ and $J_k$ are independent. Define the process $\xi^k_t \coloneqq \bm{1}_{\{\tau_{k-1} < t\}} \lambda^k_t$, $t \geq 0$, which is nonnegative and progressively measurable with respect to $\mathbb{G}^{k-1}$. By the boundedness assumption on $\lambda^k$ and the finiteness of $\tau_{k-1}$, clearly we have $\int_0^t \xi^k_s \,ds < \infty$ for all $t \geq 0$ and $\int_0^\infty \xi^k_s \,ds = \infty$, $\mathbb{P}$-a.s. Note that we can rewrite
	\begin{equation*}
		\tau_k = \inf\biggl\{t > 0: \int_0^t \xi^k_s \,ds \geq E_k\biggr\},
	\end{equation*}
	and $\mathbb{G}^k = \mathbb{G}^{k-1} \newtriangledown \mathbb{A}^k$ using Proposition~\ref{prp:tridwn}. Thus, we are in a position to apply Lemma~\ref{lem:cox_J} to deduce that the process $M^k_t = J_k \bm{1}_{\{\tau_k \leq t\}} - \mathbb{E}[J_k] \int_0^{\tau_k \land t} \xi^k_s \,ds$, $t \geq 0$, is a $\mathbb{G}^k$-martingale, and $\mathbb{G}^{k-1} \hookrightarrow \mathbb{G}^k$. Then, to show that $M^k$ is a $\mathbb{G}$-martingale, it suffices to show $\mathbb{G}^k \hookrightarrow \mathbb{G}$. However, this is a direct consequence of Lemma~\ref{lem:immerse}~(ii), once we notice that $\mathbb{G}$ can be rewritten as $\mathbb{G} = \newbigtriangledown_{i=0}^\infty \mathbb{G}^i$ using Proposition~\ref{prp:tridwn}.
	
	Finally, we show that $M$ is a $\mathbb{G}$-martingale. From what we have proved so far, we know $M$ is a $\mathbb{G}$-local martingale with a localizing sequence $(\tau_k)_{k \in \mathbb{N}}$ since $M_{\tau_k \land t} = \sum_{i=1}^k M^i_t$. To show that it is a (true) martingale, we define $N_t \coloneqq \sum_{k=1}^\infty |J_k| \bm{1}_{\{\sigma_k \leq t\}}$, which is a compound Poisson process with intensity $C$ and jump size distributed as $|J_1|$. Note that one has $|M_t| \leq N_t + C\mathbb{E}[|J_1|]t$, which implies that $M^*_t \coloneqq \sup_{s \leq t} |M_s|$ is integrable for every $t \geq 0$. This suffices to show the martingale property, and the proof is complete.
\end{proof}

\begin{remark}\label{rem:cox_h(J)s}
	Following the same argument as in Remark~\ref{rem:cox_h(J)}, under the setting of Theorem~\ref{thm:cox}, for any measurable function $h: \mathbb{R} \to \mathbb{R}$ such that $h(J_1)$ is integrable, one can show the process
	\begin{equation*}
		M^h_t \coloneqq X^h_t - \mathbb{E}[h(J_1)] \int_0^t \sum_{k=1}^\infty \bm{1}_{\{\tau_{k-1} < s \leq \tau_k\}} \lambda^k_s \,ds,\quad t \geq 0,
	\end{equation*}
	is a $\mathbb{G}$-martingale, where $X^h_t \coloneqq \sum_{k=1}^\infty h(J_k) \bm{1}_{\{\tau_k \leq t\}}$.
\end{remark}

The pure jump process $X$ constructed in Theorem~\ref{thm:cox} is an It\^o semimartingale, so we can talk about its differential characteristics. For the self-containment of this paper, we briefly recall the definition of differential characteristics, which is cited from \cite{MR4814246}, Definition~2.8.

\begin{definition}\label{def:diff_char}
	Let $X$ be an $\mathbb{R}^d$-valued It\^o semimartingale. The \emph{differential characteristics} of $X$ associated with a truncation function $h: \mathbb{R}^d \to \mathbb{R}^d$ is the triplet $(\beta, \alpha, \kappa)$ consisting in:\footnote{By truncation function, we mean that $h$ is measurable, bounded and $h(x) = x$ in a neighborhood of $0$.}
	\begin{enumerate}[label=(\roman*), nosep]
		\item $\beta$ is an $\mathbb{R}^d$-valued predictable process such that $\int_0^\cdot \beta_s \,ds$ is the predictable finite variation part of the special semimartingale $X^h \coloneqq X - \sum_{s \leq \cdot} (\Delta X_s - h(\Delta X_s))$.
		
		\item $\alpha$ is an $\mathbb{S}_+^d$-valued predictable process such that $\int_0^\cdot \alpha_s \,ds = \langle X^c, X^c \rangle \coloneqq (\langle X^{i, c}, X^{j, c} \rangle)_{1 \leq i, j \leq d}$, where $X^c = (X^{i, c})_{1 \leq i \leq d}$ is the continuous local martingale part of $X$.
		
		\item $\kappa$ is a predictable L\'evy transition kernel from $\Omega \times \mathbb{R}_+$ to $\mathbb{R}^d$ such that $\kappa_t(d\xi) dt$ is the compensator of the jump measure $\mu^X$ of $X$, where $\mu^X(dt, d\xi) \coloneqq \sum_{s > 0} \bm{1}_{\{\Delta X_s \neq 0\}} \delta_{(s, \Delta X_s)}(dt, d\xi)$ is an integer-valued random measure on $\mathbb{R}_+ \times \mathbb{R}^d$.\footnote{By L\'evy transition kernel, we mean that $\kappa_{\omega, t}(d\xi)$ is a L\'evy measure on $\mathbb{R}^d$ for each $(\omega, t) \in \Omega \times \mathbb{R}_+$, i.e.\ $\kappa_{\omega, t}(\{0\}) = 0$ and $\int_{\mathbb{R}^d} 1 \land |\xi|^2 \,\kappa_{\omega, t}(d\xi) < \infty$.}
	\end{enumerate}
\end{definition}

\begin{theorem}\label{thm:cox2charac}
	Let $X$, $\mathbb{G}$, $(J_k)_{k \in \mathbb{N}^*}$, $(\tau_k)_{k \in \mathbb{N}}$ and $(\lambda^k)_{k \in \mathbb{N}^*}$ be given by Theorem~\ref{thm:cox}. Furthermore, suppose that each $\lambda^k$ is predictable with respect to $\mathbb{G}$. Let $h: \mathbb{R} \to \mathbb{R}$ be a truncation function. Then, $X$ is a $\mathbb{G}$-It\^o semimartingale whose differential characteristics associated with $h$ are
	\begin{equation*}
		\beta_t = \mathbb{E}[h(J_1)] \lambda_t,\quad
		\alpha_t = 0,\quad
		\kappa_t(d\xi) = \lambda_t \nu(d\xi),
	\end{equation*}
	where $\nu$ is the law of $J_1$ and $\lambda_t \coloneqq \sum_{k=1}^\infty \bm{1}_{\{\tau_{k-1} < t \leq \tau_k\}} \lambda^k_t$.
\end{theorem}

\begin{proof}
	First, to obtain $\beta$, we truncate the jumps of $X$ using the truncation function $h$. As was discussed in Remark~\ref{rem:cox_h(J)s} (using the notation therein), we can decompose the special semimartingale $X^h$ into
	\begin{equation*}
		X^h_t = \int_0^t \mathbb{E}[h(J_1)] \lambda_s \,ds + M^h_t,
	\end{equation*}
	where $M^h$ is a martingale. This proves $\beta_t = \mathbb{E}[h(J_1)] \lambda_t$.
	
	Next, since $X$ is a finite variation process, its continuous local martingale part is $0$, leading to $\alpha = 0$.
	
	Lastly, we show the jump measure $\mu^X$ of $X$ has compensator $\lambda_t \nu(d\xi) dt$. According to \cite{MR1943877}, Theorem~II.1.8 and the monotone convergence theorem, it suffices to show
	\begin{equation}\label{eq:compen_meas}
		\mathbb{E}\biggl[\int_{[0, T] \times \mathbb{R}} W(u, \xi) \,\mu^X(du, d\xi)\biggr]
		= \mathbb{E}\biggl[\int_0^T \int_{\mathbb{R}} W(u, \xi) \lambda_u \,\nu(d\xi) \,du\biggr],
	\end{equation}
	for all $T > 0$ and $W: \Omega \times \mathbb{R}_+ \times \mathbb{R} \to \mathbb{R}_+$ measurable with respect to $\mathcal{P} \otimes \mathcal{B}(\mathbb{R})$, where $\mathcal{P}$ is the predictable $\sigma$-algebra on $\Omega \times \mathbb{R}_+$ with filtration $\mathbb{G} = (\mathcal{G}_t)_{t \geq 0}$. It is easy to check that \eqref{eq:compen_meas} holds for all $W$ of the forms:
	\begin{enumerate}[label=(\roman*), nosep]
		\item $\bm{1}_{A \times \{0\} \times B}$ with $A \in \mathcal{G}_0$ and $B \in \mathcal{B}(\mathbb{R})$,
		
		\item $\bm{1}_{A \times (s, t] \times B}$ with $A \in \mathcal{G}_s$, $0 \leq s < t$ and $B \in \mathcal{B}(\mathbb{R})$.
	\end{enumerate}
	Indeed, for case (ii), \eqref{eq:compen_meas} reads $\mathbb{E}[\bm{1}_A (X^{\bm{1}_B}_{t \land T} - X^{\bm{1}_B}_{s \land T})] = \mathbb{E}[\bm{1}_A \nu(B) \int_{s \land T}^{t \land T} \lambda_u \,du]$, which is equivalent to $\mathbb{E}[\bm{1}_A (M^{\bm{1}_B}_{t \land T} - M^{\bm{1}_B}_{s \land T})] = 0$, but this is simply the martingale property of $M^{\bm{1}_B}_{\cdot \land T}$. Then, Dynkin's $\pi$-$\lambda$ theorem tells us that \eqref{eq:compen_meas} holds for all $W$ of the form $\bm{1}_E$ with $E \in \mathcal{P} \otimes \mathcal{B}(\mathbb{R})$, and a standard approximation argument completes the proof. Thus, we have $\kappa_t(d\xi) = \lambda_t \nu(d\xi)$.
\end{proof}

The converse direction of Theorem~\ref{thm:cox2charac} is true if the intensity $\lambda$ is a deterministic function of time and the process $X$ itself.

\begin{theorem}\label{thm:charac2cox}
	Let $(\Omega, \mathcal{F}, \mathbb{G}, \mathbb{P})$ be a filtered probability space, where $\mathbb{G} = (\mathcal{G}_t)_{t \geq 0}$ satisfies the usual conditions. Let $\nu$ be a probability measure on $\mathbb{R}$ with $\nu(\{0\}) = 0$ and $\int_{\mathbb{R}} |\xi| \,\nu(d\xi) < \infty$. Let $\lambda: \mathbb{R}_+ \times \mathbb{R} \to \mathbb{R}_+$ be a measurable function, satisfying $c \leq \lambda \leq C$ for some constants $0 < c < C$. Suppose that $X$ is an It\^o semimartingale on $\Omega$ whose differential characteristics associated with a truncation function $h$ are given by
	\begin{equation*}
		\beta_t = \int_{\mathbb{R}} h(\xi) \,\nu(d\xi) \cdot \lambda(t, X_{t-}),\quad
		\alpha_t = 0,\quad
		\kappa_t(d\xi) = \lambda(t, X_{t-}) \nu(d\xi),
	\end{equation*}
	and $X_0 = 0$. Then, there exist a sequence of i.i.d.\ $\operatorname{Exp}(1)$ random variables $(E_k)_{k \in \mathbb{N}^*}$ and a sequence of i.i.d.\ $\nu$-distributed random variables $(J_k)_{k \in \mathbb{N}^*}$, with $(E_k)_{k \in \mathbb{N}^*}$ and $(J_k)_{k \in \mathbb{N}^*}$ being independent, such that $X_t = \sum_{k=1}^\infty J_k \bm{1}_{\{\tau_k \leq t\}}$, where $(\tau_k)_{k \in \mathbb{N}}$ is inductively defined by $\tau_0 \coloneqq 0$ and
	\begin{equation}\label{eq:st_sec2}
		\tau_k \coloneqq \inf\biggl\{t > \tau_{k-1}: \int_{\tau_{k-1}}^t \lambda(s, X_{\tau_{k-1}}) \,ds \geq E_k\biggr\},\quad
		k \in \mathbb{N}^*.
	\end{equation}
	In other words, $X$ is a doubly stochastic compound Poisson process with intensity $(\lambda(t, X_{t-}))_{t \geq 0}$ and jump size distribution $\nu$.
\end{theorem}

\begin{proof}
	First we study the structure of the jump times of $X$. Note that $\kappa_t(d\xi) dt$ is the compensator of the jump measure $\mu^X$ of $X$. For fixed $t > 0$, applying \cite{MR1943877}, Theorem~II.1.8 with $W(s, \xi) = \bm{1}_{[0, t]}(s)$, we get
	\begin{equation*}
		\mathbb{E}[\mu^X([0, t] \times \mathbb{R})]
		= \mathbb{E}\biggl[\int_0^t \lambda(s, X_{s-}) \,ds\biggr]
		\leq Ct.
	\end{equation*}
	This implies that for $\mathbb{P}$-a.s.\ $\omega \in \Omega$, the sample path $X_\cdot(\omega)$ has finitely many jumps on any finite time interval. Thus, we can inductively define $\tau_0 \coloneqq 0$ and
	\begin{equation*}
		\tau_k \coloneqq \inf\{t > \tau_{k-1}: \Delta X_t \neq 0\},\quad
		k \in \mathbb{N}^*.
	\end{equation*}
	It follows that $(\tau_k)_{k \in \mathbb{N}}$ is a sequence of stopping times that increases to infinity $\mathbb{P}$-a.s., and exhausts the jumps of $X$. Note that on the event $\{\tau_{k-1} < \infty\}$, one has strict inequality $\tau_{k-1} < \tau_k$ $\mathbb{P}$-a.s. Moreover, by an inductive argument, assuming $\mathbb{E}[\tau_{k-1}] < \infty$ and taking $W(s, \xi) = \bm{1}_{(\tau_{k-1}, \tau_k]}(s)$ in \cite{MR1943877}, Theorem~II.1.8, we deduce that
	\begin{equation*}
		c\mathbb{E}[\tau_k - \tau_{k-1}]
		\leq \mathbb{E}\biggl[\int_{\tau_{k-1}}^{\tau_k} \lambda(s, X_{s-}) \,ds\biggr]
		= \mathbb{E}[\mu^X((\tau_{k-1}, \tau_k] \times \mathbb{R})] \leq 1,
	\end{equation*}
	so $\mathbb{E}[\tau_k] < \infty$. This shows that $(\tau_k)_{k \in \mathbb{N}}$ is a sequence of integrable (thus finite) stopping times.
	
	Next, we use the canonical representation in \cite{MR1943877}, Theorem~II.2.34 to derive
	\begin{equation*}
		X_t = \int_0^t \beta_s \,ds - \int_0^t \int_{\mathbb{R}} h(\xi) \,\kappa_s(d\xi) \,ds + \int_{[0, t] \times \mathbb{R}} \xi \,\mu^X(ds, d\xi)
		= \sum_{k=1}^\infty \Delta X_{\tau_k} \bm{1}_{\{\tau_k \leq t\}},
	\end{equation*}
	which shows that $X$ is a pure jump process. Now for $k \in \mathbb{N}^*$, define $E_k \coloneqq \int_{\tau_{k-1}}^{\tau_k} \lambda(s, X_{\tau_{k-1}}) \,ds$ and $J_k \coloneqq \Delta X_{\tau_k}$. Clearly, \eqref{eq:st_sec2} is satisfied, so it only remains to prove that $(E_k)_{k \in \mathbb{N}^*}$ is i.i.d.\ $\operatorname{Exp}(1)$, $(J_k)_{k \in \mathbb{N}^*}$ is i.i.d.\ $\nu$-distributed, and these two sequences are independent.
	
	Again, by \cite{MR1943877}, Theorem~II.1.8 with $W(s, \xi) = \bm{1}_{(\tau_{k-1}, \tau_k]}(s)$, we know that the process
	\begin{equation*}
		L^k_t \coloneqq \bm{1}_{\{\tau_k \leq t\}} - \int_{\tau_{k-1} \land t}^{\tau_k \land t} \lambda(s, X_{\tau_{k-1}}) \,ds,\quad t \geq 0,
	\end{equation*}
	is a martingale. Let $u \in \mathbb{R}$ and $i$ be the imaginary unit. We then use \cite{MR1943877}, Theorem~I.4.61 to compute the exponential martingale of $-iuL^k$:
	\begin{equation}\label{eq:M}
		M^{k, u}_t \coloneqq \mathscr{E}(-iuL^k)_t
		= (1 - iu)^{\bm{1}_{\{\tau_k \leq t\}}} \exp\biggl(iu \int_{\tau_{k-1} \land t}^{\tau_k \land t} \lambda(s, X_{\tau_{k-1}}) \,ds\biggr).
	\end{equation}
	In particular, $M^{k, u}$ is a bounded martingale, so it follows that
	\begin{equation*}
		\mathbb{E}[M^{k, u}_{\tau_k} \,|\, \mathcal{G}_{\tau_{k-1}}] = M^{k, u}_{\tau_{k-1}}
		\quad\implies\quad
		\mathbb{E}[e^{iuE_k} \,|\, \mathcal{G}_{\tau_{k-1}}] = \frac{1}{1 - iu}.
	\end{equation*}
	This implies that $E_k$ is independent of $\mathcal{G}_{\tau_{k-1}}$ and has law $\operatorname{Exp}(1)$. On the other hand, let $w \in \mathbb{R}$. According to \cite{MR1943877}, Theorem~II.2.42, the process
	\begin{equation}\label{eq:N}
		N^w_t \coloneqq e^{iwX_t} - \biggl(\int_{\mathbb{R}} e^{iw\xi} \,\nu(d\xi) - 1\biggr) \cdot \int_0^t e^{iw X_{s-}} \lambda(s, X_{s-}) \,ds,\quad t \geq 0,
	\end{equation}
	is a martingale. Thus, $M^{k, u} N^w - [M^{k, u}, N^w]$ is a local martingale, and by \cite{MR1943877}, Theorem~I.4.52 we can compute
	\begin{equation}\label{eq:[M,N]}
		[M^{k, u}, N^w]_t = \bm{1}_{\{\tau_k \leq t\}} \Delta M^{k, u}_{\tau_k} \Delta N^w_{\tau_k}
		= -iu e^{iuE_k} e^{iwX_{\tau_{k-1}}} (e^{iwJ_k} - 1) \bm{1}_{\{\tau_k \leq t\}}.
	\end{equation}
	Recall that $\mathbb{E}[\tau_k] < \infty$, so it is easy to check that $M^{k, u} N^w - [M^{k, u}, N^w]$ is a uniformly integrable martingale on $[0, \tau_k]$. This gives us
	\begin{equation}\label{eq:MN-[M,N]}
		\mathbb{E}[M^{k, u}_{\tau_k} N^w_{\tau_k} - [M^{k, u}, N^w]_{\tau_k} \,|\, \mathcal{G}_{\tau_{k-1}}]
		= M^{k, u}_{\tau_{k-1}} N^w_{\tau_{k-1}} - [M^{k, u}, N^w]_{\tau_{k-1}}.
	\end{equation}
	Plugging \eqref{eq:M}, \eqref{eq:N}, \eqref{eq:[M,N]} into \eqref{eq:MN-[M,N]}, and using the fact that $E_k \sim \operatorname{Exp}(1)$ is independent of $\mathcal{G}_{\tau_{k-1}}$, we can simplify \eqref{eq:MN-[M,N]} to
	\begin{equation*}
		\mathbb{E}[e^{iuE_k + iwJ_k} \,|\, \mathcal{G}_{\tau_{k-1}}]
		= \frac{1}{1 - iu} \int_{\mathbb{R}} e^{iw\xi} \,\nu(d\xi).
	\end{equation*}
	This implies that $(E_k, J_k)$ is independent of $\mathcal{G}_{\tau_{k-1}}$ and has law $\operatorname{Exp}(1) \otimes \nu$. Given that $(E_{k-1}, J_{k-1})$ is $\mathcal{G}_{\tau_{k-1}}$-measurable, it follows that all the $E_k$'s and $J_k$'s are independent, and the proof is complete.
\end{proof}

\section{Main Results - Counting Processes}\label{sec:3}
In this section, we present our main results on inverting the Markovian projections for counting processes. By a counting process, we are referring to a jump process $X$ of the type $X_t = \sum_{k=1}^\infty \bm{1}_{\{\tau_k \leq t\}}$, with $0 < \tau_1 < \tau_2 < \cdots$ $\mathbb{P}$-a.s. In particular, the jump size of $X$ is constant $1$. In Section~\ref{sec:4}, we will extend our setting to the case of general discrete jump sizes. However, we emphasize that the current section is not merely a special case and duplication of Section~\ref{sec:4}. In this section, we impose weaker assumptions and prove stronger results. The proof technique is different and of interest in its own right. Thus, we discuss the case of counting processes separately, which also gives a clearer illustration of how the Cox construction works.

\subsection{Existence Results}
Recall our underlying filtered probability space $(\Omega, \mathcal{F}, \mathbb{F}, \mathbb{P})$ as specified in Section~\ref{subsec:cox}. First we make the following assumptions.

\begin{assumption}\label{asm:1}
	Let $0 < L < U$ be positive constants. Assume that:
	\begin{enumerate}[label=(\roman*), nosep]
		\item $\eta$ is a predictable process on the filtered probability space $(\Omega, \mathcal{F}, \mathbb{F}, \mathbb{P})$, satisfying $L \leq \eta \leq U$ $(\mathbb{P} \otimes dt)$-a.e.
		
		\item $\lambda: \mathbb{R}_+ \times \mathbb{N} \to \mathbb{R}_+$ is a measurable function, satisfying $L \leq \lambda(\cdot, x) \leq U$ Lebesgue-a.e.\ for each $x \in \mathbb{N}$.
	\end{enumerate}
\end{assumption}

Our main result of this section is stated as follows.

\begin{theorem}\label{thm:inv_mp}
	Let Assumption~\ref{asm:1} hold. Suppose that $\Omega$ supports a sequence of i.i.d.\ $\operatorname{Exp}(1)$ random variables $(E_k)_{k \in \mathbb{N}^*}$ which is independent of $\mathcal{F}_\infty$. Then, there exist an enlarged filtration $\mathbb{G}$ (satisfying the usual conditions) and a counting process $X$ adapted to $\mathbb{G}$ such that the process
	\begin{equation}\label{eq:X_compen}
		X_t - \int_0^t \frac{\eta_s}{\mathbb{E}[\eta_s \,|\, X_{s-}]} \lambda(s, X_{s-}) \,ds,\quad t \geq 0,
	\end{equation}
	is a $\mathbb{G}$-martingale.
\end{theorem}

Being a counting process, $X$ is automatically a locally integrable increasing process, so Theorem~\ref{thm:inv_mp} tells us that $\int_0^\cdot \frac{\eta_s}{\mathbb{E}[\eta_s \,|\, X_{s-}]} \lambda(s, X_{s-}) \,ds$ is the $\mathbb{G}$-compensator of $X$. The enlargement of the filtration is necessary in general. Consider the trivial case where $\eta$ is deterministic, $\lambda \equiv 1$, and $\mathbb{F} = (\sigma(\mathcal{N}^\mathbb{P}))_{t \geq 0}$ with $\mathcal{N}^\mathbb{P}$ being the set of all $\mathbb{P}$-null sets. Then, it would be impossible to construct an $\mathbb{F}$-adapted counting process (which has to be $\mathbb{P}$-a.s.\ deterministic) whose $\mathbb{F}$-compensator is $t$. Lastly, the predictability of $\eta$ is not needed in the proof of Theorem~\ref{thm:inv_mp}. The reason we impose this assumption is that we want the integrand of the compensator to be predictable. Indeed, it suffices to let $\eta$ be progressively measurable.

Now we prove Theorem~\ref{thm:inv_mp}. The main difficulty lies in the conditional expectation term $\mathbb{E}[\eta_t \,|\, X_{t-}]$. The idea of the proof is to first replace the conditional expectation with something ``known'', use the Cox construction technique to build $X$, then re-compute the conditional expectation using $X$. This leads to a fixed point problem, which can be solved using the Banach fixed point theorem.

\begin{proof}[Proof of Theorem~\ref{thm:inv_mp}]
	We divide the proof into three steps.
	
	\medskip
	\emph{Step 1: Cox construction.}\quad We pick a sequence $(\gamma_k)_{k \in \mathbb{N}} \subset L^\infty(\mathbb{R}_+; [L, U])$, which will be determined later. Let $\tau_0 \coloneqq 0$ and inductively define the random times
	\begin{equation}\label{eq:st_sec3}
		\tau_k \coloneqq \inf\biggl\{t > \tau_{k-1}: \int_{\tau_{k-1}}^t \frac{\eta_s}{\gamma_{k-1}(s)} \lambda(s, k-1) \,ds \geq E_k\biggr\},\quad
		k \in \mathbb{N}^*.
	\end{equation}
	Let $\mathbb{A}^k$ be the natural filtration of the process $(\bm{1}_{\{\tau_k \leq t\}})_{t \geq 0}$ for $k \in \mathbb{N}^*$, and define the enlarged filtration $\mathbb{G} \coloneqq \mathbb{F} \newtriangledown (\newbigtriangledown_{k=1}^\infty \mathbb{A}^k)$. We are now in a position to apply Theorem~\ref{thm:cox} with $\lambda^k_t = \frac{\eta_t}{\gamma_{k-1}(t)} \lambda(t, k-1)$ and $J_k \equiv 1$. By Assumption~\ref{asm:1}, one has $L^2/U \leq \lambda^k \leq U^2/L$ for all $k$ $(\mathbb{P} \otimes dt)$-a.e., thus all the assumptions of Theorem~\ref{thm:cox} are satisfied. Define the counting process $X_t \coloneqq \sum_{k=1}^\infty \bm{1}_{\{\tau_k \leq t\}}$, $t \geq 0$. It follows that $X$ is finite, and the process
	\begin{equation*}
		M_t \coloneqq X_t - \int_0^t \sum_{k=1}^\infty \bm{1}_{\{\tau_{k-1} < s \leq \tau_k\}} \frac{\eta_s}{\gamma_{k-1}(s)} \lambda(s, k-1) \,ds,\quad t \geq 0,
	\end{equation*}
	is a $\mathbb{G}$-martingale. The next step is to choose a proper sequence of functions $(\gamma_k)_{k \in \mathbb{N}}$ that serves our purpose.
	
	\medskip
	\emph{Step 2: Deriving the fixed point problem.}\quad Our goal is to find a sequence $(\gamma_k)_{k \in \mathbb{N}} \subset L^\infty(\mathbb{R}_+; [L, U])$ such that for each $k \in \mathbb{N}$ and Lebesgue-a.e.\ $t \geq 0$,
	\begin{equation}\label{eq:fp0}
		\gamma_k(t) = \mathbb{E}[\eta_t \,|\, X_{t-} = k].
	\end{equation}
	Note that \eqref{eq:fp0} is a fixed point problem, as the right-hand side depends on the $\gamma_k$'s through $X$. If \eqref{eq:fp0} is true, using the fact that $\{\tau_k < t \leq \tau_{k+1}\} = \{X_{t-} = k\}$, we can rewrite
	\begin{equation*}
		\begin{split}
			M_t &= X_t - \int_0^t \sum_{k=0}^\infty \bm{1}_{\{X_{s-} = k\}} \frac{\eta_s}{\mathbb{E}[\eta_s \,|\, X_{s-} = k]} \lambda(s, k) \,ds\\
			&= X_t - \int_0^t \frac{\eta_s}{\mathbb{E}[\eta_s \,|\, X_{s-}]} \lambda(s, X_{s-}) \,ds.
		\end{split}
	\end{equation*}
	In Step 1, we already proved that $M$ is a $\mathbb{G}$-martingale, so this would finish the proof of the theorem.
	
	The remaining work is to solve the fixed point problem \eqref{eq:fp0} for $k \in \mathbb{N}$. Let us compute the right-hand side of \eqref{eq:fp0} to derive a more explicit expression in $\gamma_k$:
	\begin{equation*}
		\mathbb{E}[\eta_t \,|\, X_{t-} = k]
		= \mathbb{E}[\eta_t \,|\, \tau_k < t \leq \tau_{k+1}]
		= \frac{\mathbb{E}\bigl[\eta_t \bm{1}_{\{\tau_k < t \leq \tau_{k+1}\}}\bigr]}{\mathbb{E}\bigl[\bm{1}_{\{\tau_k < t \leq \tau_{k+1}\}}\bigr]}.
	\end{equation*}
	The numerator equals
	\begin{equation}\label{eq:numerator}
		\begin{split}
			\mathbb{E}\bigl[\eta_t \bm{1}_{\{\tau_k < t \leq \tau_{k+1}\}}\bigr]
			&= \mathbb{E}\bigl[\mathbb{E}\bigl[\eta_t \bm{1}_{\{\tau_k < t\}} \bm{1}_{\{\tau_{k+1} \geq t\}} \,\big|\, \mathcal{F}_\infty, E_1, ..., E_k\bigr]\bigr]\\
			&= \mathbb{E}\biggl[\eta_t \bm{1}_{\{\tau_k < t\}} \mathbb{P}\biggl(\int_{\tau_k}^t \frac{\eta_s}{\gamma_k(s)} \lambda(s, k) \,ds \leq E_{k+1} \,\bigg|\, \mathcal{F}_\infty, E_1, ..., E_k\biggr)\biggr]\\
			&= \mathbb{E}\biggl[\eta_t \bm{1}_{\{\tau_k < t\}} \exp\biggl(-\int_{\tau_k}^t \frac{\eta_s}{\gamma_k(s)} \lambda(s, k) \,ds\biggr)\biggr],
		\end{split}
	\end{equation}
	where in the second equality we used the fact that $\tau_k$ is $(\mathcal{F}_\infty \lor \sigma(E_1, ..., E_k))$-measurable, and in the last equality we used the independence between $E_{k+1}$ and $\mathcal{F}_\infty \lor \sigma(E_1, ..., E_k)$. The denominator can be computed in the same way. Then, \eqref{eq:fp0} reads
	\begin{equation}\label{eq:fp}
		\gamma_k(t) = \frac{\mathbb{E}\bigl[\eta_t \bm{1}_{\{\tau_k < t\}} \exp\bigl(-\int_{\tau_k}^t \frac{\eta_s}{\gamma_k(s)} \lambda(s, k) \,ds\bigr)\bigr]}{\mathbb{E}\bigl[\bm{1}_{\{\tau_k < t\}} \exp\bigl(-\int_{\tau_k}^t \frac{\eta_s}{\gamma_k(s)} \lambda(s, k) \,ds\bigr)\bigr]}.
	\end{equation}
	Now for fixed $k \in \mathbb{N}$, the function $\gamma_k$ appears explicitly on the right-hand side of \eqref{eq:fp}, while the functions $\gamma_0, ..., \gamma_{k-1}$ are implicitly encoded in the definition of $\tau_k$. The crucial point is that functions $\gamma_\ell$ with $\ell > k$ are not involved in \eqref{eq:fp}. This allows us to solve \eqref{eq:fp} by induction on $k \in \mathbb{N}$.
	
	\medskip
	\emph{Step 3: Solving the fixed point problem.}\quad We now prove that there exists a unique sequence $(\gamma_k)_{k \in \mathbb{N}} \subset L^\infty(\mathbb{R}_+; [L, U])$ that solves \eqref{eq:fp}. The proof of the base case and the induction step are the same, so we prove them together. Assume that we have already found functions $\gamma_0, ..., \gamma_{m-1}$ so that $\eqref{eq:fp}$ holds for $k = 0, ..., m-1$ (if $m=0$, this assumption is vacuous, corresponding to the base case). In other words, the joint law of $(\tau_m, \eta)$ is known, which does not depend on the choice of $\gamma_m$. We prove the existence and uniqueness of $\gamma_m$.
	
	Our proof is based on the Banach fixed point theorem. We define the map $\Phi_m: L^\infty(\mathbb{R}_+; [L, U]) \to L^\infty(\mathbb{R}_+; [L, U])$ via
	\begin{equation*}
		\Phi_m(\gamma)(t) \coloneqq \frac{f_m(t; \gamma)}{g_m(t; \gamma)},\quad
		\gamma \in L^\infty(\mathbb{R}_+; [L, U]),\, t \geq 0,
	\end{equation*}
	where
	\begin{equation*}
		\begin{split}
			f_m(t; \gamma) &\coloneqq \mathbb{E}\biggl[\eta_t \bm{1}_{\{\tau_m < t\}} \exp\biggl(-\int_{\tau_m}^t \frac{\eta_s}{\gamma(s)} \lambda(s, m) \,ds\biggr)\biggr],\\
			g_m(t; \gamma) &\coloneqq \mathbb{E}\biggl[\bm{1}_{\{\tau_m < t\}} \exp\biggl(-\int_{\tau_m}^t \frac{\eta_s}{\gamma(s)} \lambda(s, m) \,ds\biggr)\biggr].
		\end{split}
	\end{equation*}
	We want to show that $\Phi_m$ has a unique fixed point. Since $L^\infty(\mathbb{R}_+; [L, U])$ is a closed subset of $L^\infty(\mathbb{R}_+; \mathbb{R})$, we know it is a complete metric space. However, the problem is that $\Phi_m$ is not a contraction map with respect to $\lVert \cdot \rVert_\infty \coloneqq \lVert \cdot \rVert_{L^\infty(\mathbb{R}_+)}$, so we need some more work.
	
	Fix a finite time horizon $T > 0$. With a little abuse of notation, we instead prove that the map $\Phi_m: L^\infty([0, T]; [L, U]) \to L^\infty([0, T]; [L, U])$ has a unique fixed point. Let $\gamma, \widetilde{\gamma} \in L^\infty([0, T]; [L, U])$. By the mean value theorem, we have the following estimate:
	\begin{equation}\label{eq:fg_diff}
		\begin{split}
			&\max\biggl\{\frac{1}{U} |f_m(t; \gamma) - f_m(t; \widetilde{\gamma})|, |g_m(t; \gamma) - g_m(t; \widetilde{\gamma})|\biggr\}\\
			&\quad\quad\leq \mathbb{E}\biggl[\bm{1}_{\{\tau_m < t\}} \biggl|\exp\biggl(-\int_{\tau_m}^t \frac{\eta_s}{\gamma(s)} \lambda(s, m) \,ds\biggr) - \exp\biggl(-\int_{\tau_m}^t \frac{\eta_s}{\widetilde{\gamma}(s)} \lambda(s, m) \,ds\biggr)\biggr|\biggr]\\
			&\quad\quad\leq \mathbb{E}\biggl[\bm{1}_{\{\tau_m < t\}} \int_{\tau_m}^t \eta_s \lambda(s, m) \biggl|\frac{1}{\gamma(s)} - \frac{1}{\widetilde{\gamma}(s)}\biggr| \,ds\biggr]\\
			&\quad\quad\leq \frac{U^2}{L^2} \mathbb{P}(\tau_m < t) \int_0^t |\gamma(s) - \widetilde{\gamma}(s)| \,ds.
		\end{split}
	\end{equation}
	We also note that
	\begin{equation}\label{eq:g_lb}
		g_m(t; \cdot) \geq e^{-(U^2/L) t} \mathbb{P}(\tau_m < t).
	\end{equation}
	Thus, using \eqref{eq:fg_diff} and \eqref{eq:g_lb}, we have that for $t \in [0, T]$,
	\begin{equation}\label{eq:phi_diff}
		\begin{split}
			&|\Phi_m(\gamma)(t) - \Phi_m(\widetilde{\gamma})(t)|\\
			&\quad\quad\leq \frac{g_m(t; \widetilde{\gamma}) |f_m(t; \gamma) - f_m(t; \widetilde{\gamma})| + f_m(t; \widetilde{\gamma}) |g_m(t; \gamma) - g_m(t; \widetilde{\gamma})|}{g_m(t; \gamma) g_m(t; \widetilde{\gamma})}\\
			&\quad\quad\leq \frac{|f_m(t; \gamma) - f_m(t; \widetilde{\gamma})| + U|g_m(t; \gamma) - g_m(t; \widetilde{\gamma})|}{g_m(t; \gamma)}\\
			&\quad\quad\leq \frac{2U^3 e^{(U^2/L) T}}{L^2} \int_0^t |\gamma(s) - \widetilde{\gamma}(s)| \,ds.
		\end{split}
	\end{equation}
	In general, the map $\Phi_m: L^\infty([0, T]; [L, U]) \to L^\infty([0, T]; [L, U])$ is still not a contraction with respect to $\lVert \cdot \rVert_{\infty, T} \coloneqq \lVert \cdot \rVert_{L^\infty([0, T])}$. What we will do next is to introduce an equivalent norm for which $\Phi_m$ is a contraction map.
	
	Let $a > 0$, and define $\lVert f \rVert_{\infty, T, a} \coloneqq \lVert e^{-a \cdot} f(\cdot) \rVert_{\infty, T}$ for $f \in L^\infty([0, T]; \mathbb{R})$. It is easy to check that $\lVert \cdot \rVert_{\infty, T, a}$ is a norm on $L^\infty([0, T]; \mathbb{R})$ which is equivalent to $\lVert \cdot \rVert_{\infty, T}$. Moreover, $L^\infty([0, T]; [L, U])$ is still a complete metric space with respect to $\lVert \cdot \rVert_{\infty, T, a}$. Then, for $\gamma, \widetilde{\gamma} \in L^\infty([0, T]; [L, U])$ and $t \in [0, T]$, we use \eqref{eq:phi_diff} to estimate
	\begin{equation*}
		\begin{split}
			&|e^{-at} \Phi_m(\gamma)(t) - e^{-at} \Phi_m(\widetilde{\gamma})(t)|
			\leq C(T) e^{-at} \int_0^t e^{as} |e^{-as} \gamma(s) - e^{-as} \widetilde{\gamma}(s)| \,ds\\
			&\quad\quad\quad\quad\leq C(T) \lVert \gamma - \widetilde{\gamma} \rVert_{\infty, t, a} \int_0^t e^{-a(t-s)} \,ds
			= \frac{C(T) (1 - e^{-at})}{a} \lVert \gamma - \widetilde{\gamma} \rVert_{\infty, t, a},
		\end{split}
	\end{equation*}
	where $C(T) > 0$ is some constant only depending on $T$ (and the universal constants $L$, $U$). This implies
	\begin{equation*}
		\lVert \Phi_m(\gamma) - \Phi_m(\widetilde{\gamma}) \rVert_{\infty, T, a}
		\leq \frac{C(T)}{a} \lVert \gamma - \widetilde{\gamma} \rVert_{\infty, T, a}.
	\end{equation*}
	By choosing $a = 2C(T)$, we get that the map $\Phi_m: L^\infty([0, T]; [L, U]) \to L^\infty([0, T]; [L, U])$ is a contraction with respect to $\lVert \cdot \rVert_{\infty, T, a}$. The Banach fixed point theorem then yields the existence and uniqueness of the fixed point $\gamma^*$ (which is a function defined on $[0, T]$).
	
	So far we have proved the map $\Phi_m: L^\infty([0, T]; [L, U]) \to L^\infty([0, T]; [L, U])$ admits a unique fixed point $\gamma^*$. By the uniqueness property and the arbitrariness of $T > 0$, this suffices to show the existence and uniqueness of $\gamma_m = \gamma^*$ that satisfies \eqref{eq:fp} on $\mathbb{R}_+$.
\end{proof}

As a consequence of the Markovian projection results in \cite{MR4814246}, we have the following corollary.

\begin{corollary}\label{cor:mp}
	Let $X$ be the counting process in Theorem~\ref{thm:inv_mp}. Then, there exists a doubly stochastic Poisson process $\widehat{X}$ with intensity $(\lambda(t, \widehat{X}_{t-}))_{t \geq 0}$ such that for every $t \geq 0$, the law of $\widehat{X}_t$ agrees with the law of $X_t$.
\end{corollary}

\begin{proof}
	We use \cite{MR4814246}, Theorem~3.2. Let $h(x) = x \bm{1}_{\{|x| \leq r\}}$, $x \in \mathbb{R}$, be a truncation function, where $0 < r < 1$. By Theorem~\ref{thm:cox2charac}, the differential characteristics of $X$ associated with $h$ are given by
	\begin{equation*}
		\beta_t = 0,\quad
		\alpha_t = 0,\quad
		\kappa_t(d\xi) = \frac{\eta_t}{\mathbb{E}[\eta_t \,|\, X_{t-}]} \lambda(t, X_{t-}) \delta_1(d\xi).
	\end{equation*}
	Taking conditional expectations $\mathbb{E}[\cdot \,|\, X_{t-}]$, we see that
	\begin{equation*}
		b(t, x) = 0,\quad
		a(t, x) = 0,\quad
		k(t, x, d\xi) = \lambda(t, x) \delta_1(d\xi)
	\end{equation*}
	satisfy (3.3) in \cite{MR4814246}. Clearly, assumptions (3.2) and (3.4) in \cite{MR4814246} are justified. Thus, it follows that there exists an It\^o semimartingale $\widehat{X}$ with differential characteristics $b(t, \widehat{X}_{t-})$, $a(t, \widehat{X}_{t-})$ and $k(t, \widehat{X}_{t-}, d\xi)$ associated with $h$ such that for every $t \geq 0$, the law of $\widehat{X}_t$ agrees with the law of $X_t$. In particular, by Theorem~\ref{thm:charac2cox}, $\widehat{X}$ is a doubly stochastic Poisson process with intensity $(\lambda(t, \widehat{X}_{t-}))_{t \geq 0}$.
\end{proof}

\subsection{Uniqueness Results}
In Theorem~\ref{thm:inv_mp}, we proved the existence of the inverted Markovian projections for counting processes using the Cox construction. A natural question to ask is whether there are other ways to construct the inversion. In other words, is the inversion unique, say in law? More precisely, let $X_t = \sum_{k=1}^\infty \bm{1}_{\{\tau_k \leq t\}}$, $t \geq 0$, be a counting process, with $0 < \tau_1 < \tau_2 < \cdots$ $\mathbb{P}$-a.s. Let $\mathbb{G}$ be a filtration with $\mathbb{F} \subseteq \mathbb{G}$, such that $X$ is adapted to $\mathbb{G}$ and the process in \eqref{eq:X_compen} is a $\mathbb{G}$-martingale. We are interested in whether such an $X$ is unique in law.

We may even consider a more stringent concept of uniqueness in law, which arises from a ``weak formulation''. Let us say we are given a function $\lambda$ satisfying Assumption~\ref{asm:1}~(ii) and a Borel probability measure $\mu$ on $L^\infty(\mathbb{R}_+; [L, U])$ equipped with the weak-* topology.\footnote{The space $L^\infty(\mathbb{R}_+; [L, U])$ is weakly-* compact and weakly-* metrizable by the Banach--Alaoglu theorem. Thus, it is a Polish space relative to the weak-* topology.} We say $X^*$ is a weak solution to the LSI model, if there exists a filtered probability space $(\Omega^*, \mathcal{F}^*, \mathbb{F}^*, \mathbb{P}^*)$, satisfying the usual conditions, that supports a predictable process $\eta^*$ with law $\mu$ and a counting process $X^*$ such that the process in \eqref{eq:X_compen} with $(\eta, X)$ replaced by $(\eta^*, X^*)$ is an $\mathbb{F}^*$-martingale. One can then ask about the uniqueness in law of such $X^*$. Note that our existence result in Theorem~\ref{thm:inv_mp} is stronger than this notion of ``weak existence''.

Recall in the proof of Theorem~\ref{thm:inv_mp}, we used the Banach fixed point theorem to show the existence and uniqueness of the fixed point solution to \eqref{eq:fp0}. From this, the uniqueness in law of the inversion $X$ seems promising. Unfortunately, we do not have uniqueness in law in general, in both senses discussed above. See the example below.

\begin{example}\label{eg:counter}
	Let $E_1 \sim \operatorname{Exp}(1)$, and define $\eta_t \coloneqq 1 + \bm{1}_{\{E_1 < t\}}$, $t \geq 0$. Let $\mathbb{F} = (\mathcal{F}_t)_{t \geq 0}$ be the natural filtration of $\eta$ (enlarged by $\mathcal{N}^{\mathbb{P}}$ and right-continuous regularized). Then, $\eta$ is predictable with respect to $\mathbb{F}$, satisfying $1 \leq \eta \leq 2$. Finally, let $\lambda(\cdot, \cdot) \equiv 1$. We will construct two different inversions of the Markovian projection, with different laws.
	
	Our first construction is of the Cox type. That is to say, by introducing a sequence of i.i.d.\ $\operatorname{Exp}(1)$ random variables $(E^*_k)_{k \in \mathbb{N}^*}$ which is independent of $\mathcal{F}_\infty$, we follow the proof of Theorem~\ref{thm:inv_mp} to construct an enlarged filtration $\mathbb{G}$ and a counting process $X_t = \sum_{k=1}^\infty \bm{1}_{\{\tau_k \leq t\}}$, $t \geq 0$, adapted to $\mathbb{G}$ such that the process in \eqref{eq:X_compen} is a $\mathbb{G}$-martingale. In particular, each $\tau_k$ satisfies \eqref{eq:st_sec3} with $E_k$ replaced by $E^*_k$, for some deterministic function $\gamma_{k-1}$.
	
	On the other hand, take a sequence of i.i.d.\ $\operatorname{Exp}(1)$ random variables $(E_k)_{k=2}^\infty$ which is independent of $E_1$. Let $\sigma_k \coloneqq \sum_{i=1}^k E_i$ for $k \in \mathbb{N}^*$, and define the counting process $\widetilde{X}_t \coloneqq \sum_{k=1}^\infty \bm{1}_{\{\sigma_k \leq t\}}$, $t \geq 0$. Let $\widetilde{\mathbb{G}}$ be the natural filtration of $\widetilde{X}$ (enlarged by $\mathcal{N}^{\mathbb{P}}$ and right-continuous regularized). Then, $\widetilde{X}$ is a $\widetilde{\mathbb{G}}$-Poisson process with intensity $1$. We claim that the process in \eqref{eq:X_compen} with $X$ replaced by $\widetilde{X}$ is a $\widetilde{\mathbb{G}}$-martingale. Indeed, from the fact that $\{\eta_t = 1\} = \{\widetilde{X}_{t-} = 0\}$, we know $\eta_t$ is $\sigma(\widetilde{X}_{t-})$-measurable. Thus, the process in \eqref{eq:X_compen} with $X$ replaced by $\widetilde{X}$ is simply $\widetilde{X}_t - t$, which is a $\widetilde{\mathbb{G}}$-martingale.
	
	The counting process $X$ from the Cox construction is clearly not a Poisson process with intensity $1$. This can be seen from the definition of $\tau_1$:
	\begin{equation*}
		\tau_1 \coloneqq \inf\biggl\{t > 0: \int_0^t \frac{\eta_s}{\gamma_0(s)} \,ds \geq E^*_1\biggr\},
	\end{equation*}
	where $\gamma_0$ is some deterministic function, and recall that $E^*_1 \sim \operatorname{Exp}(1)$ is independent of $\mathcal{F}_\infty$. Since $\eta$ is nontrivial between time $0$ and $\tau_1$, we see $\tau_1$ does not follow $\operatorname{Exp}(1)$. This shows that $X$ and $\widetilde{X}$ have different laws, so uniqueness does not hold.
\end{example}

The second construction in Example~\ref{eg:counter} seems a little bit trivial. The key problem is that the first jump time $\sigma_1$ of $\widetilde{X}$ is already a stopping time with respect to the original filtration $\mathbb{F}$, while in the Cox construction the jump times $(\tau_k)_{k \in \mathbb{N}^*}$ are driven by i.i.d.\ $\operatorname{Exp}(1)$ random variables which are also independent of $\mathcal{F}_\infty$. To rule out the cases like the second construction, we do not want the $\tau_k$'s to be $\mathbb{F}$-stopping times, i.e.\ we need some ``new'' source of randomness.

The following theorem gives a sufficient condition under which uniqueness in law holds. In other words, we have uniqueness in law within a constrained class of inversions.

\begin{theorem}\label{thm:unique}
	Let $X_t = \sum_{k=1}^\infty \bm{1}_{\{\tau_k \leq t\}}$, $t \geq 0$, be a counting process, with $0 < \tau_1 < \tau_2 < \cdots$ $\mathbb{P}$-a.s. Let $\mathbb{A}^k$ be the natural filtration of the process $(\bm{1}_{\{\tau_k \leq t\}})_{t \geq 0}$ for $k \in \mathbb{N}^*$. Let $\mathbb{G}^0 \coloneqq \mathbb{F}$, $\mathbb{G}^k \coloneqq \mathbb{F} \newtriangledown (\newbigtriangledown_{i=1}^k \mathbb{A}^i)$ for $k \in \mathbb{N}^*$ and $\mathbb{G} \coloneqq \mathbb{F} \newtriangledown (\newbigtriangledown_{i=1}^\infty \mathbb{A}^i)$. Suppose that
	\begin{enumerate}[label=(\roman*), nosep]
		\item $\mathbb{F} = \mathbb{G}^0 \hookrightarrow \cdots \hookrightarrow \mathbb{G}^{k-1} \hookrightarrow \mathbb{G}^k \hookrightarrow \cdots \hookrightarrow \mathbb{G}$;
		
		\item for each $k \in \mathbb{N}^*$, the process $Z^k_t \coloneqq \mathbb{P}(\tau_k > t \,|\, \mathcal{G}^{k-1}_t)$, $t \geq 0$, has a continuous modification which is strictly positive, where $\mathcal{G}^{k-1}_t$ is the element of $\mathbb{G}^{k-1}$ at time $t$.
	\end{enumerate}
	Moreover, let $\eta$, $\lambda$ be given by Assumption~\ref{asm:1}, and suppose that the process in \eqref{eq:X_compen} is a $\mathbb{G}$-martingale. Then, the law of $X$ is uniquely determined by $\lambda$ and the law of $\eta$.
\end{theorem}

\begin{proof}
	For $k \in \mathbb{N}$, define the function $\gamma_k \in L^\infty(\mathbb{R}_+; [L, U])$ via
	\begin{equation*}
		\gamma_k(t) \coloneqq \mathbb{E}[\eta_t \,|\, X_{t-} = k],\quad t \geq 0.
	\end{equation*}
	Then, for Lebesgue-a.e.\ $t \geq 0$, we have $\gamma_{X_{t-}}(t) = \mathbb{E}[\eta_t \,|\, X_{t-}]$, so the process
	\begin{equation*}
		M_t \coloneqq X_t - \int_0^t \frac{\eta_s}{\gamma_{X_s-}(s)} \lambda(s, X_{s-}) \,ds,\quad t \geq 0,
	\end{equation*}
	is a $\mathbb{G}$-martingale. Now we fix $k \in \mathbb{N}^*$. By stopping $M$ at the $\mathbb{G}$-stopping times $\tau_{k-1}$ and $\tau_k$, then taking the difference, we know that the process
	\begin{equation*}
		M^k_t \coloneqq \bm{1}_{\{\tau_k \leq t\}} - \int_{\tau_{k-1} \land t}^{\tau_k \land t} \frac{\eta_s}{\gamma_{k-1}(s)} \lambda(s, k-1) \,ds,\quad t \geq 0,
	\end{equation*}
	is a $\mathbb{G}^k$-martingale. Write $A^k_t \coloneqq \int_{\tau_{k-1} \land t}^{\tau_k \land t} \frac{\eta_s}{\gamma_{k-1}(s)} \lambda(s, k-1) \,ds$ for simplicity. Applying \cite{MR3204220}, Proposition~3.1 to $M^k$, and using our assumptions on $\mathbb{G}^{k-1}$, $\mathbb{G}^k$ and $Z^k$, we conclude that $A^k_{\tau_k} \sim \operatorname{Exp}(1)$ and is independent of $\mathcal{G}^{k-1}_\infty$.
	
	Denote $E_k \coloneqq A^k_{\tau_k}$ for $k \in \mathbb{N}^*$. Note that each $E_k$ is $\mathcal{G}^k_\infty$-measurable, so the sequence $(E_k)_{k \in \mathbb{N}^*}$ is i.i.d.\ $\operatorname{Exp}(1)$ and independent of $\mathcal{F}_\infty$. On the other hand, by the definition of the $E_k$'s, it is easy to check that for all $k \in \mathbb{N}^*$,
	\begin{equation*}
		\tau_k = \inf\biggl\{t > \tau_{k-1}: \int_{\tau_{k-1}}^t \frac{\eta_s}{\gamma_{k-1}(s)} \lambda(s, k-1) \,ds \geq E_k\biggr\},
	\end{equation*}
	with $\tau_0 \coloneqq 0$. This shows that the counting process $X$ is Cox-constructed using the sequence $(E_k)_{k \in \mathbb{N}^*}$. As a consequence, given $\lambda$ and $(\gamma_k)_{k \in \mathbb{N}}$, $X$ is a deterministic function of the random objects $\eta$ and $(E_k)_{k \in \mathbb{N}^*}$. More precisely, one can write $X = F_{\lambda, \gamma}(\eta, (E_k)_{k \in \mathbb{N}^*})$ for some measurable map $F_{\lambda, \gamma}: L^\infty(\mathbb{R}_+; [L, U]) \times \mathbb{R}^\mathbb{N} \to D(\mathbb{R}_+; \mathbb{R}_+)$ depending on $\lambda$ and $(\gamma_k)_{k \in \mathbb{N}}$.\footnote{Here $L^\infty(\mathbb{R}_+; [L, U])$ is equipped with the weak-* topology, and $D(\mathbb{R}_+; \mathbb{R}_+)$ is the Skorokhod space.} To get the uniqueness in law of $X$, it suffices to show the uniqueness of $(\gamma_k)_{k \in \mathbb{N}}$. However, by the definition of the $\gamma_k$'s and following exactly the same argument as in the proof of Theorem~\ref{thm:inv_mp}, one sees that $(\gamma_k)_{k \in \mathbb{N}}$ solves \eqref{eq:fp}, which has a unique solution. That is to say, $(\gamma_k)_{k \in \mathbb{N}}$ is uniquely determined by $\lambda$ and the law of $\eta$. This finishes the proof.
\end{proof}

\begin{remark}
	The assumptions (i) and (ii) in Theorem~\ref{thm:unique} are basically a characterization of $X$ being of the ``Cox'' type. Recall the second construction $\widetilde{X}$ in Example~\ref{eg:counter} which we want to rule out. We see that $\widetilde{X}$ does not satisfy (ii) in that $Z^1_t = \bm{1}_{\{\sigma_1 > t\}}$ is not continuous.
\end{remark}

\section{Main Results - General Discrete Jump Size Distributions}\label{sec:4}
In Section~\ref{sec:3}, we inverted the Markovian projections for counting processes, whose jump size is constant $1$. In this section, we extend our results to the case of general discrete jump size distributions, i.e.\ we allow the jump size to take countably many values. To be more specific, let $(J_k)_{k \in \mathbb{N}^*}$ be a sequence of i.i.d.\ non-vanishing discrete random variables. We will construct pure jump processes of the type $X_t = \sum_{k=1}^\infty J_k \bm{1}_{\{\tau_k \leq t\}}$, $t \geq 0$, with $0 < \tau_1 < \tau_2 < \cdots$ $\mathbb{P}$-a.s., whose compensator has a conditional expectation term involving $X$ itself.

Recall the proof of Theorem~\ref{thm:inv_mp}, in which we solved the fixed point problem \eqref{eq:fp0}. In order to compute the conditional expectation $\mathbb{E}[\eta_t \,|\, X_{t-} = k]$, we used the fact that $\{X_{t-} = k\} = \{\tau_k < t \leq \tau_{k+1}\}$. In other words, $X_t = k$ if and only if $X$ has jumped exactly $k$ times before time $t$. Thus, the fixed point equation for $\gamma_k$ only involves $\gamma_0, ..., \gamma_{k-1}$ (and $\gamma_k$). This allows us to solve for the $\gamma_k$'s sequentially by induction. The key point here is the simple structure of the sample paths of $X$, which are non-decreasing step functions.

However, in the context of general jump sizes, especially when both positive and negative jump sizes are allowed, the previous proof breaks. In this case, the sample paths of $X$ are no longer monotone. To better illustrate this, let us consider the simple case where $X$ has jump sizes $\pm 1$. Then, the event $\{X_{t-} = n\}$ does not contain much information about how many times $X$ has jumped before time $t$. Indeed, for $n > 0$, we have
\begin{equation*}
	\{X_{t-} = n\} = \bigcup_{\ell = 0}^\infty \bigl(\{\tau_{n + 2\ell} < t \leq \tau_{n + 2\ell + 1}\} \cap \{(J_k)_{k=1}^{n+2\ell} \text{ has } \ell \text{ entries} =-1\}\bigr).
\end{equation*}
For a particular $\ell \geq 0$, the set $\{\tau_{n + 2\ell} < t \leq \tau_{n + 2\ell + 1}\}$ says that $X$ has $n + 2\ell$ jumps before time $t$, and the set $\{(J_k)_{k=1}^{n+2\ell} \text{ has } \ell \text{ entries} =-1\}$ further says that among these $n + 2\ell$ jumps, $n + \ell$ of them are upward and $\ell$ of them are downward. This means on the event $\{X_{t-} = n\}$, $X$ can have arbitrarily many jumps, and can hit any natural number, before time $t$. As a consequence, the fixed point equation $\gamma_n = \mathbb{E}[\eta_t \,|\, X_{t-} = n]$ involves all the other $\gamma_m$'s. Thus, we cannot solve for the $\gamma_n$'s sequentially and inductively. Instead, we need to solve for the $\gamma_n$'s simultaneously as a solution to a system of infinitely many fixed point equations.

Due to the more complicated structure of our system of fixed point equations, we will use different tools and require some regularity conditions to prove our main results. Before we do so, let us first introduce some notation.

\begin{definition}
	We use the following notation.
	\begin{enumerate}[label=(\roman*), nosep]
		\item Let $a \in \mathbb{R}^n$. Define $S_k(a) \coloneqq \sum_{i=1}^k a_i$ for $0 \leq k \leq n$. In particular, $S_0(a) = 0$ by our convention.
		
		\item Let $\varnothing \neq A \subseteq \mathbb{R}$. The collection of all finite (including empty) sums of elements of $A$ is denoted by $\operatorname{FS}(A)$, namely
		\begin{equation*}
			\operatorname{FS}(A) \coloneqq \{S_k(a): a \in A^k,\, k \in \mathbb{N}\},
		\end{equation*}
		where $A^0 \coloneqq \{0\}$ by convention.
	\end{enumerate}
\end{definition}

Intuitively speaking, suppose $X$ is a pure jump process with initial value $0$ whose jump sizes are taken from the set $A$. Then, $\operatorname{FS}(A)$ consists of the places that $X$ can end up in after any finite number of jumps occurring. A simple fact worth mentioning is that if $A$ is a countable set, so is $\operatorname{FS}(A)$. Next, we make the following assumptions.

\begin{assumption}\label{asm:2}
	Let $0 < L < U$ be positive constants. Assume that:
	\begin{enumerate}[label=(\roman*), nosep]
		\item $\nu$ is a discrete probability measure on $\mathbb{R}$, satisfying $\nu(\{0\}) = 0$ and $\int_{\mathbb{R}} |\xi| \,\nu(d\xi) < \infty$. Let $\operatorname{Atom}(\nu)$ denote the set of points $x \in \mathbb{R}$ with $\nu(\{x\}) > 0$.
		
		\item $\eta$ is a predictable process on the filtered probability space $(\Omega, \mathcal{F}, \mathbb{F}, \mathbb{P})$, satisfying $L \leq \eta \leq U$ $(\mathbb{P} \otimes dt)$-a.e. Moreover, there exist $\alpha \in (0, 1]$ and for each $T > 0$ a constant $C_T > 0$ such that
		\begin{equation*}
			\mathbb{E}[|\eta_t - \eta_s|]
			\leq C_T |t-s|^\alpha,\quad
			\forall\, s, t \in [0, T].
		\end{equation*}
		
		\item $\lambda: \mathbb{R}_+ \times \operatorname{FS}(\operatorname{Atom}(\nu)) \to \mathbb{R}_+$ is a measurable function, satisfying $L \leq \lambda(\cdot, x) \leq U$ Lebesgue-a.e.\ for each $x \in \operatorname{FS}(\operatorname{Atom}(\nu))$.
	\end{enumerate}
\end{assumption}

The measure $\nu$ will be the jump size distribution of the pure jump process $X$ (which starts from $0$), so the set of all possible values that $X$ can take is exactly $\operatorname{FS}(\operatorname{Atom}(\nu))$. On the other hand, compared to Assumption~\ref{asm:1}, here we impose an extra regularity condition on $\eta$, which is the local H\"older continuity of sample paths in the average sense. This condition is satisfied by a broad class of processes. For instance, when $\eta$ is a continuous It\^o semimartingale of the form $\eta = \int_0^\cdot \beta_s \,ds + \int_0^\cdot \sigma_s \,dB_s$ with $\mathbb{E}[|\beta_t| + |\sigma_t|^2]$ locally bounded, the Burkholder--Davis--Gundy inequality guarantees that $\eta$ satisfies Assumption~\ref{asm:2} with $\alpha = 1/2$. Now we state the main results of this section.

\begin{theorem}\label{thm:inv_mp_ex}
	Let Assumption~\ref{asm:2} hold. Suppose that $\Omega$ supports a sequence of i.i.d.\ $\operatorname{Exp}(1)$ random variables $(E_k)_{k \in \mathbb{N}^*}$ and a sequence of i.i.d.\ $\nu$-distributed random variables $(J_k)_{k \in \mathbb{N}^*}$ such that $\mathcal{F}_\infty$, $(E_k)_{k \in \mathbb{N}^*}$ and $(J_k)_{k \in \mathbb{N}^*}$ are independent. Then, there exist an enlarged filtration $\mathbb{G}$ (satisfying the usual conditions) and a $\mathbb{G}$-adapted pure jump process $X$ starting from $0$ whose jump measure has $\mathbb{G}$-compensator
	\begin{equation}\label{eq:compen}
		\frac{\eta_t}{\mathbb{E}[\eta_t \,|\, X_{t-}]} \lambda(t, X_{t-}) \nu(d\xi) dt.
	\end{equation}
\end{theorem}

The basic idea of proving Theorem~\ref{thm:inv_mp_ex} is similar to that of Theorem~\ref{thm:inv_mp}, which is based on the Cox construction and some suitable fixed point theorem. As mentioned before, we now have a more complicated system of fixed point equations. To tackle this new difficulty, we will use the Schauder fixed point theorem as our main tool. Before we go into the proof, let us recall the following basic results from calculus.

\begin{lemma}\label{lem:calc1}
	Let $\alpha \in (0, 1]$. Suppose that $f$, $g$ are $\alpha$-H\"older continuous functions on an interval $I$, and $|f| \leq C$, $c \leq |g| \leq C$ for some constants $0 < c < C$. Then, $f/g$ is $\alpha$-H\"older continuous on $I$, and $|f/g|_{C^{0, \alpha}(I)}$ only depends on $|f|_{C^{0, \alpha}(I)}$, $|g|_{C^{0, \alpha}(I)}$, $c$ and $C$.
\end{lemma}

\begin{lemma}\label{lem:calc2}
	Suppose that $(f_n)_{n \in \mathbb{N}}$, $(g_n)_{n \in \mathbb{N}}$ converge uniformly to $f$, $g$ respectively on an interval $I$, and $|f_n| \leq C$, $c \leq |g_n| \leq C$, $\forall\, n \in \mathbb{N}$, for some constants $0 < c < C$. Then, $(f_n/g_n)_{n \in \mathbb{N}}$ converges uniformly to $f/g$ on $I$.
\end{lemma}

The proofs of these two lemmas are elementary, so we omit them. Now we prove the theorem.

\begin{proof}[Proof of Theorem~\ref{thm:inv_mp_ex}]
	We divide the proof into five steps.
	
	\medskip
	\emph{Step 1: Cox construction.}\quad Denote $\mathcal{A} \coloneqq \operatorname{Atom}(\nu)$ for short. Since $\operatorname{FS}(\mathcal{A})$ is a countable set, we enumerate it as $(x_n)_{n \in \mathbb{N}}$. The sequence $(x_n)_{n \in \mathbb{N}}$ would be all the possible values that our target process $X$ takes. Define the function space
	\begin{equation*}
		\mathcal{C} \coloneqq \{f \in C((0, \infty)): L \leq f \leq U\},
	\end{equation*}
	and we pick a sequence $(\gamma_{x_n})_{n \in \mathbb{N}} \subset \mathcal{C}$, which will be determined later.\footnote{Here we intentionally use $(x_n)_{n \in \mathbb{N}}$, instead of $\mathbb{N}$, as the index set. As we will see in the proof, $\gamma_{x_n}(t)$ represents $\mathbb{E}[\eta_t \,|\, X_{t-} = x_n]$. It would be more convenient to label the $\gamma$'s using the level of $X$. If we use $\mathbb{N}$ as the index set, in some equations we would need to introduce a new notation for the map $x_n \mapsto n$.} Let $\tau_0 \coloneqq 0$ and inductively define the random times
	\begin{equation*}
		\tau_k \coloneqq \inf\Biggl\{t > \tau_{k-1}: \int_{\tau_{k-1}}^t \sum_{n=0}^\infty \bm{1}_{\{\sum_{i=1}^{k-1} J_i = x_n\}} \frac{\eta_s}{\gamma_{x_n}(s)} \lambda(s, x_n) \,ds \geq E_k\Biggr\},\quad k \in \mathbb{N}^*.
	\end{equation*}
	Let $\mathbb{A}^k$ be the natural filtration of the process $(J_k \bm{1}_{\{\tau_k \leq t\}})_{t \geq 0}$ for $k \in \mathbb{N}^*$. Define the enlarged filtrations $\mathbb{G}^0 \coloneqq \mathbb{F}$, $\mathbb{G}^k \coloneqq \mathbb{F} \newtriangledown (\newbigtriangledown_{i=1}^k \mathbb{A}^i)$ for $k \in \mathbb{N}^*$ and $\mathbb{G} \coloneqq \mathbb{F} \newtriangledown (\newbigtriangledown_{i=1}^\infty \mathbb{A}^i)$. We are now in a position to apply Theorem~\ref{thm:cox} with
	\begin{equation*}
		\lambda^k_t = \sum_{n=0}^\infty \bm{1}_{\{\sum_{i=1}^{k-1} J_i \bm{1}_{\{\tau_i \leq t\}} = x_n\}} \frac{\eta_t}{\gamma_{x_n}(t)} \lambda(t, x_n),
	\end{equation*}
	which is progressively measurable with respect to $\mathbb{G}^{k-1}$. Define the pure jump process $X_t \coloneqq \sum_{k=1}^\infty J_k \bm{1}_{\{\tau_k \leq t\}}$, $t \geq 0$. It follows that $X$ is well-defined, and the process
	\begin{equation*}
		M_t \coloneqq X_t - \mathbb{E}[J_1] \int_0^t \sum_{k=1}^\infty \bm{1}_{\{\tau_{k-1} < s \leq \tau_k\}} \sum_{n=0}^\infty \bm{1}_{\{\sum_{i=1}^{k-1} J_i = x_n\}} \frac{\eta_s}{\gamma_{x_n}(s)} \lambda(s, x_n) \,ds,\quad t \geq 0,
	\end{equation*}
	is a $\mathbb{G}$-martingale.
	
	\medskip
	\emph{Step 2: Deriving the fixed point problem.}\quad Our goal is to find a sequence $(\gamma_{x_n})_{n \in \mathbb{N}} \subset \mathcal{C}$ such that for each $n \in \mathbb{N}$ and $t > 0$,
	\begin{equation}\label{eq:fp0_J}
		\gamma_{x_n}(t) = \mathbb{E}[\eta_t \,|\, X_{t-} = x_n].
	\end{equation}
	If \eqref{eq:fp0_J} is true, using the fact that $X_{t-} = \sum_{i=1}^k J_i$ on the event $\{\tau_k < t \leq \tau_{k+1}\}$, we can rewrite
	\begin{equation*}
		\begin{split}
			M_t &= X_t - \mathbb{E}[J_1] \int_0^t \sum_{k=0}^\infty \bm{1}_{\{\tau_k < s \leq \tau_{k+1}\}} \sum_{n=0}^\infty \bm{1}_{\{X_{s-} = x_n\}} \frac{\eta_s}{\mathbb{E}[\eta_s \,|\, X_{s-} = x_n]} \lambda(s, x_n) \,ds\\
			&= X_t - \mathbb{E}[J_1] \int_0^t \frac{\eta_s}{\mathbb{E}[\eta_s \,|\, X_{s-}]} \lambda(s, X_{s-}) \,ds.
		\end{split}
	\end{equation*}
	In Step 1, we already proved that $M$ is a $\mathbb{G}$-martingale. Thus, by Theorem~\ref{thm:cox2charac}, the jump measure of $X$ has $\mathbb{G}$-compensator \eqref{eq:compen}, and this would finish the proof of the theorem.
	
	The remaining work is to solve the fixed point problem \eqref{eq:fp0_J} for $n \in \mathbb{N}$. Let us compute the right-hand side of \eqref{eq:fp0_J} to derive a more explicit expression in $\gamma_{x_n}$:
	\begin{equation}\label{eq:cond_exp}
		\mathbb{E}[\eta_t \,|\, X_{t-} = x_n]
		= \frac{\mathbb{E}\bigl[\eta_t \bm{1}_{\{X_{t-} = x_n\}}\bigr]}{\mathbb{E}\bigl[\bm{1}_{\{X_{t-} = x_n\}}\bigr]}.
	\end{equation}
	For $a \in \mathcal{A}^k$, $k \in \mathbb{N}$, let $\tau^a_0 \coloneqq 0$ and inductively define the random times\footnote{Recall $\mathcal{A} \coloneqq \operatorname{Atom}(\nu)$. Here we can see the reason why we label the $\gamma$'s using $(x_n)_{n \in \mathbb{N}}$ instead of $\mathbb{N}$. Otherwise, we would need to find the index of $S_{i-1}(a)$ in the sequence $(x_n)_{n \in \mathbb{N}}$.}
	\begin{equation}\label{eq:st}
		\tau^a_i \coloneqq \inf\biggl\{t > \tau^a_{i-1}: \int_{\tau^a_{i-1}}^t \frac{\eta_s}{\gamma_{S_{i-1}(a)}(s)} \lambda(s, S_{i-1}(a)) \,ds \geq E_i\biggr\},\quad i = 1, ..., k+1.
	\end{equation}
	Intuitively speaking, $\tau^a_i$ is defined as if we ``freeze'' the first $k$ jumps of $X$: $(J_1, ..., J_k) = a$. By comparing the definitions of $\tau^a_i$ and $\tau_i$, it is easy to see that they coincide on the paths where $(J_1, ..., J_k) = a$. Then for $t > 0$, we can write (up to $\mathbb{P}$-null sets)
	\begin{equation}\label{eq:set_decomp}
		\begin{split}
			\{X_{t-} = x_n\}
			&= \bigcup_{k=0}^\infty \{\tau_k < t \leq \tau_{k+1},\, J_1 + \cdots + J_k = x_n\}\\
			&= \bigcup_{k=0}^\infty \bigcup_{\substack{a \in \mathcal{A}^k \\ S_k(a) = x_n}} \{\tau_k < t \leq \tau_{k+1},\, (J_1, ..., J_k) = a\}\\
			&= \bigcup_{k=0}^\infty \bigcup_{\substack{a \in \mathcal{A}^k \\ S_k(a) = x_n}} \{\tau^a_k < t \leq \tau^a_{k+1},\, (J_1, ..., J_k) = a\}.
		\end{split}
	\end{equation}
	For $a \in \mathcal{A}$, denote $p_a \coloneqq \nu(\{a\}) = \mathbb{P}(J_1 = a)$. Using \eqref{eq:set_decomp} and the independence argument in \eqref{eq:numerator}, we can similarly compute the numerator of \eqref{eq:cond_exp}:
	\begin{equation*}
		\begin{split}
			\mathbb{E}\bigl[\eta_t \bm{1}_{\{X_{t-} = x_n\}}\bigr]
			&= \sum_{k=0}^\infty \sum_{\substack{a \in \mathcal{A}^k \\ S_k(a) = x_n}} \prod_{i=1}^k p_{a_i} \mathbb{E}\bigl[\eta_t \bm{1}_{\{\tau^a_k < t \leq \tau^a_{k+1}\}}\bigr]\\
			&= \sum_{k=0}^\infty \sum_{\substack{a \in \mathcal{A}^k \\ S_k(a) = x_n}} \prod_{i=1}^k p_{a_i} \mathbb{E}\biggl[\eta_t \bm{1}_{\{\tau^a_k < t\}} \exp\biggl(-\int_{\tau^a_k}^t \frac{\eta_s}{\gamma_{x_n}(s)} \lambda(s, x_n) \,ds\biggr)\biggr].
		\end{split}
	\end{equation*}
	The denominator of \eqref{eq:cond_exp} can be computed in the same way. Then, \eqref{eq:fp0_J} reads
	\begin{equation}\label{eq:fp_J}
		\gamma_{x_n}(t)
		= \frac{\sum_{k=0}^\infty \sum_{\substack{a \in \mathcal{A}^k \\ S_k(a) = x_n}} \prod_{i=1}^k p_{a_i} \mathbb{E}\bigl[\eta_t \bm{1}_{\{\tau^a_k < t\}} \exp\bigl(-\int_{\tau^a_k}^t \frac{\eta_s}{\gamma_{x_n}(s)} \lambda(s, x_n) \,ds\bigr)\bigr]}{\sum_{k=0}^\infty \sum_{\substack{a \in \mathcal{A}^k \\ S_k(a) = x_n}} \prod_{i=1}^k p_{a_i} \mathbb{E}\bigl[\bm{1}_{\{\tau^a_k < t\}} \exp\bigl(-\int_{\tau^a_k}^t \frac{\eta_s}{\gamma_{x_n}(s)} \lambda(s, x_n) \,ds\bigr)\bigr]}.
	\end{equation}
	Now \eqref{eq:fp_J} looks a bit more explicit in $\gamma_{x_n}$. However, note that all the random times $\tau^a_k$'s, for $k \in \mathbb{N}$ and $a \in \mathcal{A}^k$ with $S_k(a) = x_n$, appear on the right-hand side of \eqref{eq:fp_J}, which means that all the $\gamma_{x_m}$'s can possibly be involved in the equation. This is basically due to the fact that before $X$ hits $x_n$, it can hit any values in $\operatorname{FS}(\mathcal{A})$. Thus, we cannot mimic the induction proof in Theorem~\ref{thm:inv_mp}. We need to solve the system of infinitely many fixed point equations simultaneously.
	
	\medskip
	\emph{Step 3: Solving the fixed point problem.}\quad We now prove that there exists a sequence $(\gamma_{x_n})_{n \in \mathbb{N}} \subset \mathcal{C}$ that solves \eqref{eq:fp_J}. Our proof is based on the Schauder fixed point theorem, so first we need to find a proper topological vector space to work with.
	
	Consider the space $C((0, \infty))$ equipped with the topology of uniform convergence on compact intervals. We know that $C((0, \infty))$ is a metrizable topological vector space. Clearly, $\mathcal{C}$ is a convex closed subset of $C((0, \infty))$. Now we consider the countably infinite product space $C((0, \infty))^\mathbb{N}$ equipped with the product topology. Then, $C((0, \infty))^\mathbb{N}$ is again a metrizable topological vector space (see e.g.\ \cite{MR2378491}, Theorem~3.36 and Theorem~5.2), and $\mathcal{C}^\mathbb{N}$ is a convex closed subset of $C((0, \infty))^\mathbb{N}$. For a generic element $\gamma \in C((0, \infty))^\mathbb{N}$, we denote it as $\gamma = (\gamma_{x_n})_{n \in \mathbb{N}}$, i.e.\ the coordinates are indexed by $(x_n)_{n \in \mathbb{N}}$.
	
	We define the map $\Phi: \mathcal{C}^\mathbb{N} \to \mathcal{C}^\mathbb{N}$ via
	\begin{equation*}
		\Phi(\gamma)_{x_n}(t) \coloneqq \frac{f_n(t; \gamma)}{g_n(t; \gamma)},\quad
		\gamma \in \mathcal{C}^\mathbb{N},\, n \in \mathbb{N},\, t > 0,
	\end{equation*}
	where
	\begin{equation*}
		\begin{split}
			f_n(t; \gamma) &\coloneqq \sum_{k=0}^\infty \sum_{\substack{a \in \mathcal{A}^k \\ S_k(a) = x_n}} \prod_{i=1}^k p_{a_i} f_{n, k, a}(t; \gamma),\\
			f_{n, k, a}(t; \gamma) &\coloneqq \mathbb{E}\biggl[\eta_t \bm{1}_{\{\tau^a_k(\gamma) < t\}} \exp\biggl(-\int_{\tau^a_k(\gamma)}^t \frac{\eta_s}{\gamma_{x_n}(s)} \lambda(s, x_n) \,ds\biggr)\biggr],\\
			g_n(t; \gamma) &\coloneqq \sum_{k=0}^\infty \sum_{\substack{a \in \mathcal{A}^k \\ S_k(a) = x_n}} \prod_{i=1}^k p_{a_i} g_{n, k, a}(t; \gamma),\\
			g_{n, k, a}(t; \gamma) &\coloneqq \mathbb{E}\biggl[\bm{1}_{\{\tau^a_k(\gamma) < t\}} \exp\biggl(-\int_{\tau^a_k(\gamma)}^t \frac{\eta_s}{\gamma_{x_n}(s)} \lambda(s, x_n) \,ds\biggr)\biggr].\\
		\end{split}
	\end{equation*}
	Here $\tau^a_k(\gamma)$ is defined via \eqref{eq:st}, and we emphasize its dependency on $\gamma$. Note that a priori it is not clear that $\Phi(\gamma)_{x_n}$ is a continuous function. Once established, it would immediately follow that $\Phi(\mathcal{C}^\mathbb{N}) \subseteq \mathcal{C}^\mathbb{N}$ by Assumption~\ref{asm:2}~(ii). Our goal is to prove that $\Phi$ admits a fixed point, which would complete the proof of the theorem. To apply the Schauder fixed point theorem, we need to show that $\Phi$ is continuous and $\Phi(\mathcal{C}^\mathbb{N})$ is pre-compact.
	
	\medskip
	\emph{Step 4: $\Phi(\mathcal{C}^\mathbb{N}) \subseteq \mathcal{C}^\mathbb{N}$ and $\Phi(\mathcal{C}^\mathbb{N})$ is pre-compact.}\quad Fix $n \in \mathbb{N}$ and $0 < t_0 < T$. First we show that $\Phi(\gamma)_{x_n}$ is $\alpha$-H\"older continuous on the interval $[t_0, T]$, uniformly in $\gamma$. Our proof will utilize Lemma~\ref{lem:calc1}.
	
	We begin our analysis with the $f_n(t; \cdot)$ term. Let $\gamma \in \mathcal{C}^\mathbb{N}$, $k \in \mathbb{N}$ and $a \in \mathcal{A}^k$ with $S_k(a) = x_n$. Without loss of generality, let us assume $k \geq 1$ (the case $k=0$ is similar and simpler, so we will not repeat it). From the independence between $E_k$ and $\mathcal{F}_\infty \lor \sigma(E_1, ..., E_{k-1})$, we observe that the conditional probability density function of $\tau^a_k(\gamma)$ given $\mathcal{F}_\infty, E_1, ..., E_{k-1}$ is
	\begin{equation*}
		s \mapsto \bm{1}_{\{\tau^a_{k-1}(\gamma) < s\}} \frac{\eta_s \lambda(s, S_{k-1}(a))}{\gamma_{S_{k-1}(a)}(s)} \exp\biggl(-\int_{\tau^a_{k-1}(\gamma)}^s \frac{\eta_u \lambda(u, S_{k-1}(a))}{\gamma_{S_{k-1}(a)}(u)} \,du\biggr),
	\end{equation*}
	so it follows that
	\begin{equation*}
		\begin{split}
			f_{n, k, a}(t; \gamma)
			&= \mathbb{E}\biggl[\mathbb{E}\biggl[\eta_t \bm{1}_{\{\tau^a_k(\gamma) < t\}} \exp\biggl(-\int_{\tau^a_k(\gamma)}^t \frac{\eta_s}{\gamma_{x_n}(s)} \lambda(s, x_n) \,ds\biggr) \,\bigg|\, \mathcal{F}_\infty, E_1, ..., E_{k-1}\biggr]\biggr]\\
			&= \mathbb{E}\biggl[\eta_t \int_0^t \bm{1}_{\{\tau^a_{k-1}(\gamma) < s\}} \frac{\eta_s \lambda(s, S_{k-1}(a))}{\gamma_{S_{k-1}(a)}(s)}\\ &\quad\quad\quad\quad\quad\cdot \exp\biggl(-\int_{\tau^a_{k-1}(\gamma)}^s \frac{\eta_u \lambda(u, S_{k-1}(a))}{\gamma_{S_{k-1}(a)}(u)} \,du - \int_s^t \frac{\eta_u \lambda(u, x_n)}{\gamma_{x_n}(u)} \,du\biggr) \,ds\biggr].
		\end{split}
	\end{equation*}
	We may further use the conditional law of $\tau^a_{k-1}(\gamma)$ given $\mathcal{F}_\infty, E_1, ..., E_{k-2}$ to continue the computation (if $k \geq 2$).  By repeating the above type of argument for $k$ times, we can derive
	\begin{equation*}
		\begin{split}
			f_{n, k, a}(t; \gamma)
			&= \mathbb{E}\Biggl[\eta_t \int_0^t \int_0^{s_k} \cdots \int_0^{s_2} \prod_{i=1}^k \frac{\eta_{s_i} \lambda(s_i, S_{i-1}(a))}{\gamma_{S_{i-1}(a)}(s_i)}\\ &\quad\,\,\cdot \exp\Biggl(-\sum_{i=1}^k \int_{s_{i-1}}^{s_i} \frac{\eta_u \lambda(u, S_{i-1}(a))}{\gamma_{S_{i-1}(a)}(u)} \,du - \int_{s_k}^t \frac{\eta_u \lambda(u, x_n)}{\gamma_{x_n}(u)} \,du\Biggr) \,ds_1 \cdots ds_{k-1} ds_k\Biggr],
		\end{split}
	\end{equation*}
	with $s_0 \coloneqq 0$ by convention. Note that the integrand of the iterated integral above involves the time variable $t$, so we want to factor it out. Let us rewrite the exponential term in the above equation as
	\begin{equation*}
		\begin{split}
			&\exp\Biggl(-\sum_{i=1}^k \int_{s_{i-1}}^{s_i} \frac{\eta_u \lambda(u, S_{i-1}(a))}{\gamma_{S_{i-1}(a)}(u)} \,du - \biggl(\int_{s_k}^T - \int_t^T\biggr) \frac{\eta_u \lambda(u, x_n)}{\gamma_{x_n}(u)} \,du\Biggr)\\
			&\quad\quad= \exp\Biggl(-\sum_{i=1}^{k+1} \int_{s_{i-1}}^{s_i} \frac{\eta_u \lambda(u, S_{i-1}(a))}{\gamma_{S_{i-1}(a)}(u)} \,du\Biggr) \cdot \exp\biggl(\int_t^T \frac{\eta_u \lambda(u, x_n)}{\gamma_{x_n}(u)} \,du\biggr),
		\end{split}
	\end{equation*}
	using the convention $s_{k+1} \coloneqq T$ and the fact that $S_k(a) = x_n$. This leads to the expression $f_{n, k, a}(t; \gamma) = \mathbb{E}[\eta_t \theta^{n, \gamma}_t \xi^{k, a, \gamma}_t]$, where
	\begin{equation*}
		\begin{split}
			\theta^{n, \gamma}_t &\coloneqq \exp\biggl(\int_t^T \frac{\eta_u \lambda(u, x_n)}{\gamma_{x_n}(u)} \,du\biggr),\\
			\xi^{k, a, \gamma}_t &\coloneqq \int_0^t \int_0^{s_k} \cdots \int_0^{s_2} \prod_{i=1}^k \frac{\eta_{s_i} \lambda(s_i, S_{i-1}(a))}{\gamma_{S_{i-1}(a)}(s_i)}\\ &\quad\quad\quad\quad\quad\quad\quad\quad\cdot \exp\Biggl(-\sum_{i=1}^{k+1} \int_{s_{i-1}}^{s_i} \frac{\eta_u \lambda(u, S_{i-1}(a))}{\gamma_{S_{i-1}(a)}(u)} \,du\Biggr) \,ds_1 \cdots ds_{k-1} ds_k.
		\end{split}
	\end{equation*}
	Now we notice that the integrand of the iterated integral above does not depend on $t$ any more. Then, it is easy to see that for $t \in [t_0, T]$,
	\begin{equation}\label{eq:theta_xi_est}
		\begin{split}
			1 &\leq \theta^{n, \gamma}_t \leq \exp\biggl(\frac{U^2}{L}(T-t_0)\biggr),\\
			\biggl(\frac{L^2}{U}\biggr)^k \exp\biggl(-\frac{U^2}{L} T\biggr) \frac{t_0^k}{k!} &\leq \xi^{k, a, \gamma}_t \leq \biggl(\frac{U^2}{L}\biggr)^k \exp\biggl(-\frac{L^2}{U} T\biggr) \frac{T^k}{k!}.
		\end{split}
	\end{equation}
	Thus, we have that
	\begin{equation}\label{eq:f_est}
		\frac{c^k}{k!} \leq f_{n, k, a}(t; \gamma) \leq \frac{C^k}{k!},\quad
		\forall\, t \in [t_0, T],
	\end{equation}
	where $c = c(t_0, T) > 0$ and $C = C(t_0, T) > 0$ are some constants (we omit the dependency on the universal constants $L$, $U$). On the other hand, for $t_0 \leq s < t \leq T$, we have the estimates
	\begin{equation}\label{eq:theta_diff}
		|\theta^{n, \gamma}_t - \theta^{n, \gamma}_s|
		\leq \exp\biggl(\frac{U^2}{L}(T-t_0)\biggr) \int_s^t \frac{\eta_u \lambda(u, x_n)}{\gamma_{x_n}(u)} \,du
		\leq \exp\biggl(\frac{U^2}{L}(T-t_0)\biggr) \frac{U^2}{L} |t-s|,
	\end{equation}
	where the first inequality follows from the mean value theorem, and
	\begin{equation}\label{eq:xi_diff}
		\begin{split}
			|\xi^{k, a, \gamma}_t - \xi^{k, a, \gamma}_s|
			&\leq \int_s^t \int_0^{s_k} \cdots \int_0^{s_2} \biggl(\frac{U^2}{L}\biggr)^k \exp\biggl(-\frac{L^2}{U}T\biggr) \,ds_1 \cdots ds_{k-1} ds_k\\
			&\leq \biggl(\frac{U^2}{L}\biggr)^k \exp\biggl(-\frac{L^2}{U} T\biggr) \frac{T^{k-1}}{(k-1)!} |t-s|.
		\end{split}
	\end{equation}
	Thus, by \eqref{eq:theta_xi_est}, \eqref{eq:theta_diff}, \eqref{eq:xi_diff}, and Assumption~\ref{asm:2}~(ii) (note that $\alpha \in (0, 1]$), we have
	\begin{equation}\label{eq:f_diff}
		\begin{split}
			&|f_{n, k, a}(t; \gamma) - f_{n, k, a}(s; \gamma)|\\
			&\quad\quad\leq \mathbb{E}[|\eta_t - \eta_s| \theta^{n, \gamma}_t \xi^{k, a, \gamma}_t] + \mathbb{E}[\eta_s |\theta^{n, \gamma}_t - \theta^{n, \gamma}_s| \xi_t] + \mathbb{E}[\eta_s \theta_s |\xi^{k, a, \gamma}_t - \xi^{k, a, \gamma}_s|]\\
			&\quad\quad\leq \frac{C^k}{(k-1)!} |t-s|^\alpha,\quad \forall\, s, t \in [t_0, T],
		\end{split}
	\end{equation}
	where $C = C(t_0, T) > 0$ is some constant (we omit the dependency on the universal constants $L$, $U$, $\alpha$). Now using \eqref{eq:f_est}, \eqref{eq:f_diff}, and the inequality\footnote{Here the summation starts from $k=1$. This suffices because there is at most one term with $k=0$ in the definition of $f_n(t; \gamma)$.}
	\begin{equation}\label{eq:series}
		\sum_{k=1}^\infty \sum_{\substack{a \in \mathcal{A}^k \\ S_k(a) = x_n}} \prod_{i=1}^k p_{a_i} \frac{C^k}{(k-1)!}
		= \sum_{k=1}^\infty \mathbb{P}\Biggl(\sum_{i=1}^k J_i = x_n\Biggr) \frac{C^k}{(k-1)!}
		\leq \sum_{k=1}^\infty \frac{C^k}{(k-1)!}
		< \infty,
	\end{equation}
	we deduce that there exists a constant $C = C(t_0, T) > 0$ such that
	\begin{equation}\label{eq:f}
		|f_n(t; \gamma)| \leq C
		\quad\text{and}\quad
		|f_n(t; \gamma) - f_n(s; \gamma)| \leq C|t-s|^\alpha,\quad
		\forall\, s, t \in [t_0, T].
	\end{equation}
	
	For the $g_n(t; \cdot)$ term, following the same argument, one can show that \eqref{eq:f_est}, \eqref{eq:f_diff} and \eqref{eq:f} are also satisfied by $g_{n, k, a}$ and $g_n$ respectively. Moreover, since $x_n \in \operatorname{FS}(\mathcal{A})$, there exist $k_n \in \mathbb{N}$ and $a_n \in \mathcal{A}^{k_n}$ such that $S_{k_n}(a_n) = x_n$. This implies that for $t \in [t_0, T]$,
	\begin{equation*}
		g_n(t; \gamma) \geq \prod_{i=1}^{k_n} p_{(a_n)_i} g_{n, k_n, a_n}(t; \gamma)
		\geq \prod_{i=1}^{k_n} p_{(a_n)_i} \frac{c(t_0, T)^{k_n}}{k_n!}
		\eqqcolon c(n, t_0, T) > 0.
	\end{equation*}
	To sum up, there exist constants $c = c(n, t_0, T)$ and $C = C(t_0, T)$ such that
	\begin{equation}\label{eq:g}
		c \leq |g_n(t; \gamma)| \leq C
		\quad\text{and}\quad
		|g_n(t; \gamma) - g_n(s; \gamma)| \leq C|t-s|^\alpha,\quad
		\forall\, s, t \in [t_0, T].
	\end{equation}
	
	Therefore, combining \eqref{eq:f}, \eqref{eq:g} and Lemma~\ref{lem:calc1}, we conclude that there exists a constant $\widetilde{C}(n, t_0, T) > 0$ such that
	\begin{equation*}
		|\Phi(\gamma)_{x_n}(t) - \Phi(\gamma)_{x_n}(s)|
		\leq \widetilde{C}(n, t_0, T) |t-s|^\alpha,\quad
		\forall\, s, t \in [t_0, T].
	\end{equation*}
	This proves our claim that $\Phi(\gamma)_{x_n}$ is $\alpha$-H\"older continuous on the interval $[t_0, T]$, uniformly in $\gamma$. Now we define, for $n \in \mathbb{N}$,
	\begin{equation*}
		\mathcal{K}_{x_n}
		\coloneqq \{f \in \mathcal{C}: |f|_{C^{0, \alpha}([t_0, T])} \leq \widetilde{C}(n, t_0, T),\, \forall\, 0 < t_0 < T\}.
	\end{equation*}
	What we have proved so far are the following set inclusions:
	\begin{equation*}
		\Phi(\mathcal{C}^\mathbb{N})
		\subseteq \prod_{n \in \mathbb{N}} \mathcal{K}_{x_n}
		\subseteq \mathcal{C}^\mathbb{N}.
	\end{equation*}
	By the Arzel\`a--Ascoli theorem and a standard diagonalization argument, it is easy to see that each $\mathcal{K}_{x_n}$ is pre-compact in $C((0, \infty))$. Then by Tychonoff's theorem, $\prod_{n \in \mathbb{N}} \mathcal{K}_{x_n}$ is pre-compact in $C((0, \infty))^\mathbb{N}$. Thus, $\Phi(\mathcal{C}^\mathbb{N})$ is pre-compact.
	
	\medskip
	\emph{Step 5: $\Phi$ is continuous.}\quad Let $(\gamma^m)_{m \in \mathbb{N}} \subset \mathcal{C}^\mathbb{N}$ be a sequence converging to $\gamma \in \mathcal{C}^\mathbb{N}$, i.e.\ $\gamma^m_{x_n} \to \gamma_{x_n}$ in $\mathcal{C}$ as $m \to \infty$ for each $n \in \mathbb{N}$. Our goal is to prove $\Phi(\gamma^m) \to \Phi(\gamma)$ in $\mathcal{C}^\mathbb{N}$, or equivalently, $\Phi(\gamma^m)_{x_n} \to \Phi(\gamma)_{x_n}$ in $\mathcal{C}$ as $m \to \infty$ for each $n \in \mathbb{N}$. Fix $n \in \mathbb{N}$ and $0 < t_0 < T$. It suffices to show that $(\Phi(\gamma^m)_{x_n})_{m \in \mathbb{N}}$ converges uniformly to $\Phi(\gamma)_{x_n}$ on the interval $[t_0, T]$. Our proof will utilize Lemma~\ref{lem:calc2}.
	
	First we show $f_n(\cdot; \gamma^m) \to f_n(\cdot; \gamma)$ uniformly on $[t_0, T]$ as $m \to \infty$. Consider the following estimate
	\begin{equation*}
		\sup_{t \in [t_0, T]} |f_n(t; \gamma^m) - f_n(t; \gamma)|
		\leq \sum_{k=0}^\infty \sum_{\substack{a \in \mathcal{A}^k \\ S_k(a) = x_n}} \prod_{i=1}^k p_{a_i} \sup_{t \in [t_0, T]} |f_{n, k, a}(t; \gamma^m) - f_{n, k, a}(t; \gamma)|.
	\end{equation*}
	Note that from \eqref{eq:f_est} we have
	\begin{equation*}
		\sup_{t \in [t_0, T]} |f_{n, k, a}(t; \gamma^m) - f_{n, k, a}(t; \gamma)|
		\leq \frac{2C(t_0, T)^k}{k!},\quad
		\forall\, m \in \mathbb{N}.
	\end{equation*}
	Then by \eqref{eq:series} and the dominated convergence theorem, it suffices to show for all $k \in \mathbb{N}$ and $a \in \mathcal{A}^k$ with $S_k(a) = x_n$,
	\begin{equation}\label{eq:unif_conv}
		\lim_{m \to \infty} \sup_{t \in [t_0, T]} |f_{n, k, a}(t; \gamma^m) - f_{n, k, a}(t; \gamma)| = 0.
	\end{equation}
	Again, without loss of generality, let us assume $k \geq 1$ (the case $k=0$ is similar and simpler, so we omit it). From \eqref{eq:theta_xi_est}, we have
	\begin{equation}\label{eq:tri_ineq}
		\begin{split}
			&\sup_{t \in [t_0, T]} |f_{n, k, a}(t; \gamma^m) - f_{n, k, a}(t; \gamma)|\\
			&\quad\quad\leq \sup_{t \in [t_0, T]} \mathbb{E}\bigl[\eta_t \bigl|\theta^{n, \gamma^m}_t - \theta^{n, \gamma}_t\bigr| \xi^{k, a, \gamma^m}_t\bigr] + \sup_{t \in [t_0, T]} \mathbb{E}\bigl[\eta_t \theta^{n, \gamma}_t \bigl|\xi^{k, a, \gamma^m}_t - \xi^{k, a, \gamma}_t\bigr|\bigr]\\
			&\quad\quad\leq C(k, t_0, T) \mathbb{E}\biggl[\sup_{t \in [t_0, T]} \bigl|\theta^{n, \gamma^m}_t - \theta^{n, \gamma}_t\bigr| + \sup_{t \in [t_0, T]} \bigl|\xi^{k, a, \gamma^m}_t - \xi^{k, a, \gamma}_t\bigr|\biggr].
		\end{split}
	\end{equation}
	For the $\theta$ term we use the mean value theorem to get
	\begin{equation*}
		\sup_{t \in [t_0, T]} \bigl|\theta^{n, \gamma^m}_t - \theta^{n, \gamma}_t\bigr|
		\leq \exp\biggl(\frac{U^2}{L} (T - t_0)\biggr) \int_{t_0}^T \eta_u \lambda(u, x_n) \biggl|\frac{1}{\gamma^m_{x_n}(u)} - \frac{1}{\gamma_{x_n}(u)}\biggr| \,du,
	\end{equation*}
	and for the $\xi$ term we have
	\begin{equation*}
		\begin{split}
			&\sup_{t \in [t_0, T]} \bigl|\xi^{k, a, \gamma^m}_t - \xi^{k, a, \gamma}_t\bigr|\\
			&\quad\quad\leq U^{2k} \int_0^T \int_0^{s_k} \cdots \int_0^{s_2} \Biggl|\prod_{i=1}^k \frac{1}{\gamma^m_{S_{i-1}(a)}(s_i)} \cdot \exp\Biggl(-\sum_{i=1}^{k+1} \int_{s_{i-1}}^{s_i} \frac{\eta_u \lambda(u, S_{i-1}(a))}{\gamma^m_{S_{i-1}(a)}(u)} \,du\Biggr)\\ &\quad\quad\quad\quad\quad\quad\quad- \prod_{i=1}^k \frac{1}{\gamma_{S_{i-1}(a)}(s_i)} \cdot \exp\Biggl(-\sum_{i=1}^{k+1} \int_{s_{i-1}}^{s_i} \frac{\eta_u \lambda(u, S_{i-1}(a))}{\gamma_{S_{i-1}(a)}(u)} \,du\Biggr)\Biggr| \,ds_1 \cdots ds_{k-1} ds_k.
		\end{split}
	\end{equation*}
	A straightforward application of the dominated convergence theorem yields
	\begin{equation}\label{eq:unif_conv_exp}
		\lim_{m \to \infty} \mathbb{E}\biggl[\sup_{t \in [t_0, T]} \bigl|\theta^{n, \gamma^m}_t - \theta^{n, \gamma}_t\bigr| + \sup_{t \in [t_0, T]} \bigl|\xi^{k, a, \gamma^m}_t - \xi^{k, a, \gamma}_t\bigr|\biggr] = 0.
	\end{equation}
	Chaining \eqref{eq:tri_ineq} and \eqref{eq:unif_conv_exp} gives us \eqref{eq:unif_conv}. Thus, we obtain that $f_n(\cdot; \gamma^m) \to f_n(\cdot; \gamma)$ uniformly on $[t_0, T]$ as $m \to \infty$.
	
	Following the same argument, one can show that $g_n(\cdot; \gamma^m) \to g_n(\cdot; \gamma)$ uniformly on $[t_0, T]$ as $m \to \infty$. Therefore, together with \eqref{eq:f} and \eqref{eq:g}, we use Lemma~\ref{lem:calc2} to conclude that $\Phi(\gamma^m)_{x_n} \to \Phi(\gamma)_{x_n}$ uniformly on $[t_0, T]$ as $m \to \infty$.
\end{proof}

Analogous to Corollary~\ref{cor:mp}, we have the following corollary.

\begin{corollary}
	Let $X$ be the pure jump process in Theorem~\ref{thm:inv_mp_ex}. Then, there exists a doubly stochastic compound Poisson process $\widehat{X}$ with intensity $(\lambda(t, \widehat{X}_{t-}))_{t \geq 0}$ and jump size distribution $\nu$ such that for every $t \geq 0$, the law of $\widehat{X}_t$ agrees with the law of $X_t$.
\end{corollary}

\begin{proof}
	The proof is almost the same as Corollary~\ref{cor:mp}. Again, we use \cite{MR4814246}, Theorem~3.2. Let $h(x) = x \bm{1}_{\{|x| \leq 1\}}$, $x \in \mathbb{R}$, be a truncation function. By Theorem~\ref{thm:cox2charac}, the differential characteristics of $X$ associated with $h$ are given by
	\begin{equation*}
		\beta_t = \int_{[-1, 1]} \xi \,\nu(d\xi) \cdot \frac{\eta_t}{\mathbb{E}[\eta_t \,|\, X_{t-}]} \lambda(t, X_{t-}),\quad
		\alpha_t = 0,\quad
		\kappa_t(d\xi) = \frac{\eta_t}{\mathbb{E}[\eta_t \,|\, X_{t-}]} \lambda(t, X_{t-}) \nu(d\xi).
	\end{equation*}
	Taking conditional expectations $\mathbb{E}[\cdot \,|\, X_{t-}]$, we have
	\begin{equation*}
		b(t, x) = \int_{[-1, 1]} \xi \,\nu(d\xi) \cdot \lambda(t, x),\quad
		a(t, x) = 0,\quad
		k(t, x, d\xi) = \lambda(t, x) \nu(d\xi).
	\end{equation*}
	It is easy to justify assumptions (3.2) and (3.4) in \cite{MR4814246}. Thus, Theorem~3.2 in \cite{MR4814246} and Theorem~\ref{thm:charac2cox} yield the desired results.
\end{proof}

We make a few comments on the uniqueness in law of the inversion $X$. In general, there is no uniqueness in law result, even if we impose similar assumptions as in Theorem~\ref{thm:unique}. The reason is that in the proof of Theorem~\ref{thm:inv_mp_ex}, we used the Schauder fixed point theorem to get the existence of $(\gamma_{x_n})_{n \in \mathbb{N}}$, while the uniqueness is not guaranteed. So far, to the best of our knowledge, it is not clear whether \eqref{eq:fp_J} admits a unique fixed point solution or not. Proving the uniqueness using other fixed point techniques, or constructing counterexamples where uniqueness breaks, is one of the possible directions for future work.

Another direction of interest for future work is to further extend the jump size distribution to continuous laws with finite first moment. Some discretization and approximation techniques may come into play.

Lastly, one may also consider relaxing the boundedness assumptions $L \leq \eta \leq U$ and/or $L \leq \lambda \leq U$. For instance, when $\eta$ is only assumed to be positive (and satisfies some integrability condition so that everything is still well-defined), a natural idea is to truncate $\eta$ by defining
\begin{equation*}
	\eta^n \coloneqq \frac{1}{n} \lor (\eta \land n),\quad
	n \in \mathbb{N}^*.
\end{equation*}
Then, for each $n$ one may apply Theorem~3.2 (resp.\ Theorem~4.3) to construct a counting (resp.\ pure jump) process $X^n$ whose intensity is $\frac{\eta^n_t}{\mathbb{E}[\eta^n_t \,|\, X^n_{t-}]} \lambda(t, X^n_{t-})$ (resp.\ jump measure has compensator $\frac{\eta^n_t}{\mathbb{E}[\eta^n_t \,|\, X^n_{t-}]} \lambda(t, X^n_{t-}) \nu(d\xi) dt$). However, even if one can show the tightness of the sequence $(\eta^n, X^n)_{n \in \mathbb{N}^*}$, the main difficulty lies in the step of passing to the limit, due to the notorious discontinuity (with respect to weak convergence) of the map from joint laws to conditional laws. Such an extension to the case of more general $\eta$ and/or $\lambda$ would also be interesting and worth exploring in future work.

\bibliography{bibliography}
\bibliographystyle{abbrv}

\end{document}